\documentclass{amsproc}
\usepackage{amssymb,amsfonts}

\theoremstyle{plain}
\newtheorem{corollary}{Corollary}
\newtheorem{lemma}{Lemma}
\newtheorem{proposition}{Proposition}
\newtheorem{theorem}{Theorem}
\theoremstyle{definition}
\newtheorem{definition}{Definition}
\newtheorem{notation}{Notation}
\newtheorem{remark}{Remark}
\numberwithin{equation}{section}

\begin{document}
\title[Absolute continuity]{Characterizing the absolute continuity of the
convolution of orbital measures in a classical Lie algebra}
\author{Sanjiv Kumar Gupta}
\address{Dept.\ of Mathematics and Statistics\\
Sultan Qaboos University\\
P.O.~Box 36 Al Khodh 123\\
Sultanate of Oman}
\email{gupta@squ.edu.om}
\author{Kathryn E. Hare}
\address{Dept.\ of Pure Mathematics\\
University of Waterloo\\
Waterloo, Ont.,~Canada\\
N2L 3G1}
\email{kehare@uwaterloo.ca}
\thanks{The first author would like to thank the Dept.\ of Pure Mathematics
at the University of Waterloo and the second author the School of
Mathematics and Statistics at St.~Andrews University for their hospitality
while some of this research was done. This research was supported in part by
the Edinburgh Math. Society, NSERC and Sultan Qaboos University}
\maketitle

\begin{abstract}
Let $\mathfrak{g}$ be a compact, simple Lie algebra of dimension $d$. It is
a classical result that the convolution of any $d$ non-trivial, $G$%
-invariant, orbital measures is absolutely continuous with respect to
Lebesgue measure on $\mathfrak{g}$ and the sum of any $d$ non-trivial orbits
has non-empty interior. The number $d$ was later reduced to the rank of the
Lie algebra (or rank $+1$ in the case of type $A_{n}$). More recently, the
minimal integer $k=k(X)$ such that the $k$-fold convolution of the orbital
measure supported on the orbit generated by $X$ is an absolutely continuous
measure was calculated for each $X\in \mathfrak{g}$.

In this paper $\mathfrak{g}$ is any of the classical, compact, simple Lie
algebras. We characterize the tuples $(X_{1},\dots,X_{L})$, with $X_{i}\in 
\mathfrak{g},$ which have the property that the convolution of the $L$%
-orbital measures supported on the orbits generated by the $X_{i}$ is
absolutely continuous and, equivalently, the sum of their orbits has
non-empty interior. The characterization depends on the Lie type of $%
\mathfrak{g}$ and the structure of the annihilating roots of the $X_{i}$.
Such a characterization was previously known only for type $A_{n}$.
\end{abstract}

\section{Introduction}

Let $G$ be a compact, connected simple Lie group and $\mathfrak{g}$ its Lie
algebra. Given $X\in \mathfrak{g}$, we let $\mu _{X}$ denote the $G$
-invariant, orbital measure supported on $O_{X}$, the orbit generated by $X$
under the adjoint action of $G$. Geometric properties of the Lie algebra
ensure that if a suitable number of non-trivial orbits are added together
the resulting subset of $\mathfrak{g}$ has non-empty interior and if a
suitable number of orbital measures are convolved together, the resulting
measure is absolutely continuous with respect to the Lebesgue measure on $%
\mathfrak{g}$. From the work of Ragozin in \cite{Ra2} it can be seen that
the dimension of the Lie algebra is a `suitable number'.

In a series of papers (see \cite{GHMathZ} and \cite{HJY} and the papers
cited therein) the authors, with various coauthors, improved upon Ragozin's
result determining, for each $X\in \mathfrak{g,}$ the integer $k(X)$ with
the property that $\mu _{X}^{k}$ is absolutely continuous for all $k\geq
k(X) $ and $\mu _{X}^{k}$ is singular to Lebesgue measure otherwise (where $%
\mu _{X}^{k}$ denotes the $k$-fold convolution). Furthermore, the $k$-fold
sum of $O_{X}$ has non-empty interior if $k\geq k(X)$ and otherwise has
measure zero. A formula was given for $k(X)$ depending on combinatorial
properties of the annihilating roots of $X$. In particular, it was shown
that the convolution of any $r$ orbital measures is absolutely continuous if
and only if $r$ is at least the rank of the Lie algebras when $\mathfrak{g}$
is of type $B_{n},C_{n}$ or $D_{n}$ and $r$ is at least rank+1 for the Lie
algebras of type $A_{n}$. The proofs relied heavily upon representation
theory and harmonic analysis.

By taking a geometric approach, Wright in \cite{Wr} extended these results
in the special case of the classical Lie algebra $\mathfrak{g}=su(n)$ (type $%
A_{n-1}$), proving that $\mu _{X_{1}}\ast \cdots \ast \mu _{X_{L}}$ is
absolutely continuous with respect to Lebesgue measure if and only if $%
\sum_{i=1}^{L}s_{i}\geq n(L-1)$ where $s_{i}$ is the dimension of the
largest eigenspace of the $n\times n$ matrix $X_{i},$ provided it is not the
case that $L=2$, $n\geq 4$ is even, and $X_{1},X_{2}$ each have two distinct
eigenvalues, both of multiplicity $n/2$.

Using primarily algebraic methods, Gracyzk and Sawyer (c.f., \cite{GSJFA}, 
\cite{GSLie}), addressed analogous problems in the setting of a non-compact,
symmetric space, improving upon other work of Ragozin, \cite{Ra}. In
particular, they characterized when the convolution of two (possibly
different) bi-invariant measures is absolutely continuous in the symmetric
spaces $sl(n,F)/su(n,F)$ (where the restricted root system is also type $%
A_{n-1}$).

Inspired by their methods, in this paper we characterize the $L$-tuples, $%
(X_{1},\dots ,X_{L})$ with $X_{i}\in \mathfrak{g}$, such that the
convolution $\mu _{X_{1}}\ast \cdots \ast \mu _{X_{L}}$ is absolutely
continuous when the Lie algebra is any one of the classical Lie algebras
(those of type $A_{n}$, $B_{n}$, $C_{n}$ or $D_{n}),$ leaving only one pair
in $D_{n}$ where we have been unable to decide the answer. As well, this
characterizes the $L$-tuples such that $\sum_{i=1}^{L}O_{X_{i}}$ has
non-empty interior in $\mathfrak{g}$ as opposed to measure zero. As Wright
found with type $A_{n}$, the characterization can be expressed most simply
as a function of the dimensions of the largest eigenspaces of the $X_{i}$
when these are viewed as matrices in the classical matrix Lie algebras (see
Section 3 for the precise statement). The characterization can also be
described in terms of the root structure of the set of annihilating roots of
the $X_{i},$ as was done in the previous study of convolutions of a single
orbital measure. Our argument is completely different from that used by
Wright and from the harmonic analysis - representation theory approach used
by the authors previously. It relies heavily upon the (algebraic) Lie theory
of roots and root vectors.

Using these results, we also obtain a similar characterization of the
absolute continuity of the convolution products of $G$-invariant measures, $%
\mu _{x_{i}},$ supported on conjugacy classes $C_{x_{i}}$ in $G,$ for the
elements $x_{i}\in G$ whose annihilating roots agree with those of a
preimage of $x_{i}$ in $\mathfrak{g}$ under the exponential map. This
extends work of \cite{GHAdv} where the minimal integer $k(x)$ with the
property that $\mu _{x}^{k(x)}$ is absolutely continuous was determined.

In a future paper, we will adapt our general strategy to improve upon
Gracyzk and Sawyer's symmetric space results.

Finding the density function, or Radon Nikodym derivative, of the absolutely
continuous measure $\mu _{X_{1}}\ast \cdots \ast \mu _{X_{L}}$ is a
challenging problem. In the case of the convolution of two orbital measures
in $su(n),$ this has been computed in \cite{FG}. A general formula for the
convolution of two orbital measures in terms of the projection of such
measures to maximal tori was found in \cite{DRW}. The density function for
the analogous problem on non-compact symmetric spaces was studied in \cite%
{GSJGA} (and see also the references cited there). In \cite{KT}, the sum of
two adjoint orbits in $su(n)$ is explicitly described in terms of a system
of linear equations, but for more than $2$-fold sums this too seems very
difficult. Other work investigating the smoothness properties of
convolutions of measures supported on manifolds whose product has non-empty
interior was carried out by Ricci and Stein in \cite{RS} and \cite{RS1}.

The paper is organized as follows: In section 2 we review background
material in Lie theory and introduce basic notation. In section 3 we state
the main result. The necessity of our characterization is proven in section
4. In section 5 we establish the general strategy for tackling the absolute
continuity problem and then complete the proof of the main theorem in
section 6. In section 7 we discuss consequences of our result and deduce the
absolute continuity result for convolutions of orbital measures on Lie
groups mentioned above.

\section{Notation and Background}

\subsection{Notation}

We begin by establishing notation and reviewing basic facts about roots and
root vectors. Assume $G_{n}$ is a classical, compact, connected simple Lie
group of rank $n$, one of type $A_{n}$, $B_{n}$, $C_{n}$ or $D_{n} $. We
denote by $\mathfrak{g}_{n}$ its (real) Lie algebra, $\mathfrak{t}_{n} $ a
maximal torus of $\mathfrak{g}_{n}$ and $W$ the Weyl group.

We write $[\cdot ,\cdot ]$ for the Lie bracket action. The map $ad: 
\mathfrak{g}_{n}\rightarrow \mathfrak{g}_{n}$ is given by $ad(X)(Y)=[X,Y]$.
The exponential function, $\exp ,$ is a surjection of $\mathfrak{g}_{n}$
onto $G_{n},$ and $G_{n}$ acts on $\mathfrak{g}_{n}$ by the adjoint action,
denoted $Ad(\cdot )$. Recall that for $M\in \mathfrak{g}_{n}$, 
\begin{equation*}
Ad(\exp M)=\exp (ad(M))=Id+\sum_{k=1}^{\infty }\frac{ad^{k}(M)}{k!}
\end{equation*}
where $ad^{k}(M)$ is the $k$-fold composition of $ad(M)$.

By an orbit of an element $X\in \mathfrak{g}_{n}$, we mean the subset 
\begin{equation*}
O_{X}:=\{Ad(g)(X):g\in G_{n}\}\subseteq \mathfrak{g}_{n}.
\end{equation*}
There is no loss in assuming $X$ belongs to $\mathfrak{t}_{n}$ since every
orbit contains a torus element. Orbits are compact manifolds of proper
dimension in $\mathfrak{g}_{n}$ and hence of Lebesgue measure zero. If $X=0$%
, then $O_{X}=\{0\}$ is a singleton, but otherwise $O_{X}$ has positive
dimension.

By the orbital measure, $\mu _{X},$ we mean the probability measure
invariant under the adjoint action of $G_{n}$ and compactly supported on $%
O_{X}$. It integrates bounded, continuous functions $f$ on $\mathfrak{g}_{n}$
by the rule 
\begin{equation*}
\int_{\mathfrak{g}_{n}}fd\mu _{X}=\int_{G_{n}}f(Ad(g)X)dg
\end{equation*}
where $dg$ is the Haar measure on $G_{n}$. The orbital measures are singular
to Lebesgue measure since their supports have Lebesgue measure zero. Except
in the special case when $X=0,$ $\mu _{X}$ is an example of a continuous
measure, meaning the $\mu _{X}$ - measure of any singleton is zero.

The classical Lie groups and algebras are said to be of type $A_{n}$ for $%
n\geq 1$, $B_{n}$ for $n\geq 2$, $C_{n}$ for $n\geq 3$ or $D_{n}$ for $n\geq
4$. This means that the root system of the complexified Lie algebra with
respect to the complexified torus, denoted $\Phi _{n}$, is of that Lie type.
It is often convenient to refer to type $A_{n}$ as type $SU(n+1)$ for
reasons that will become clear later.

For the convenience of the reader we describe $\Phi _{n}$ below for each of
the classical types. Note that by $e_{j}$ we mean the $j^{\prime }$th
standard basis vector of $\mathbb{R}^{n}$ (or in $\mathbb{R}^{n+1}$ in the
case of type $A_{n})$. The real span of $\Phi _{n}$, denoted $sp\Phi _{n}$,
is equal to $\mathbb{R}^{n}$ (or the subspace of $\mathbb{R}^{n+1}$ spanned
by the standard vectors $e_{j}-e_{n+1}$ for $j=1,\dots,n$ in the case of
type $A_{n}$). 
\begin{equation*}
\frame{$%
\begin{array}{cc}
\text{Lie algebra} & \text{Root system }\Phi _{n} \\ 
A_{n} & \{\pm (e_{i}-e_{j}):1\leq i<j\leq n+1\} \\ 
B_{n} & \{\pm e_{i},\pm e_{i}\pm e_{j}:1\leq i\neq j\leq n\} \\ 
C_{n} & \{\pm 2e_{i},\pm e_{i}\pm e_{j}:1\leq i\neq j\leq n\} \\ 
D_{n} & \{\pm e_{i}\pm e_{j}:1\leq i\neq j\leq n\}%
\end{array}
$}
\end{equation*}

In the case of type $A_{n}$, the Weyl group is the group of permutations on
the letters $\{1,\dots,n+1\}$. For types $B_{n}$, $C_{n}$ (and $D_{n})$, the
Weyl groups are the group of permutations on $\{1,\dots,n\}$, together with
(an even number of) sign changes.

These Lie algebras and groups can be identified with the classical matrix
algebras and groups listed below. All compact, connected simple Lie groups
are homomorphic images by finite subgroups of these classical matrix groups.

\begin{itemize}
\item $su(n)$ -- the set of $n\times n$ skew-Hermitian, trace zero matrices
is the model we use for the Lie algebra of type $A_{n-1}$. $SU(n)$ - the $%
n\times n$ special unitary matrices is a compact Lie group of type $A_{n-1}$.

\item $so(p)$ -- the set of $p\times p$ real, skew-symmetric matrices. When $%
p=2n$ it is the Lie algebra of type $D_{n}$ and when $p=2n+1$ it is of type $%
B_{n}$. $SO(p)$ - the $p\times p$ special orthogonal matrices are associated
compact Lie groups.

\item $sp(n)$ -- the set of $2\,n\times 2n$ matrices of the form $%
\begin{bmatrix}
A & B \\ 
-\overline{B} & \overline{A}%
\end{bmatrix}%
$ where $A,B$ are complex $n\times n$ matrices with $B$ symmetric and $A$
skew-Hermitian is the Lie algebra of type $C_{n}$. The $n$'th order
symplectic group, $Sp(n),$ is the set of $2n\times 2n$ unitary matrices $U$
satisfying $U^{tr}JU=J,$ where $J=%
\begin{bmatrix}
0 & -I \\ 
I & 0%
\end{bmatrix}%
$ with $I$ being the $n\times n$ identity matrix. $Sp(n)$ is a compact Lie
group of type $C_{n}.$
\end{itemize}

For each root $\alpha \in \Phi _{n},$ we let $E_{\alpha }$ denote a
corresponding root vector so that if $H\in \mathfrak{t}_{n}$, then 
\begin{equation}
[H,E_{\alpha }]=i\alpha (H)E_{\alpha }.  \label{adroot}
\end{equation}
(We make the convention that roots are real valued.) We will choose a
collection of root vectors, $\{E_{\alpha }\},$ that form a Weyl basis (see 
\cite[p.~290]{Va}). In particular, this ensures that if $\alpha $, $\beta $
and $\alpha +\beta $ are roots, then there are non-zero scalars $N_{\alpha
,\beta }$ satisfying $N_{\alpha ,\beta }=N_{-\alpha ,-\beta }$ and 
\begin{equation*}
[E_{\alpha },E_{\beta }]=N_{\alpha ,\beta }E_{\alpha +\beta }.
\end{equation*}
If $\alpha +\beta $ is not a root, then $[E_{\alpha },E_{\beta }]=0$.

The root vector, $E_{\alpha },$ can be written in a unique way as $E_{\alpha
}=RE_{\alpha }+iIE_{\alpha }$, where $RE_{\alpha }$ and $IE_{\alpha }$ both
belong to the (real) Lie algebra $\mathfrak{g}_{n}$. We refer to these as
the real and imaginary parts of the root vector. We write $FE_{\alpha }$ if
we mean either $RE_{\alpha }$ or $IE_{\alpha }$. One can easily see that $%
E_{-\alpha }=RE_{\alpha }-iIE_{\alpha }$. Furthermore, $RE_{\alpha
}=(E_{\alpha }+E_{-\alpha })/2$ and $IE_{\alpha }=(E_{\alpha }-E_{-\alpha
})/(2i)$.

The vector space spanned by $RE_{\alpha }$ and $IE_{\alpha }$ over various
sets of roots $\alpha $ will be important to us. In particular, we put 
\begin{equation}
\mathcal{V}_{n}=\{RE_{\alpha },IE_{\alpha }:\alpha \in \Phi
_{n}^{+}\}\subseteq \mathfrak{g}_{n}  \label{Vn}
\end{equation}
where $\Phi _{n}^{+}$ denotes the subset of positive roots. With this
notation the Lie algebra can be decomposed as 
\begin{equation*}
\mathfrak{g}_{n}=\mathfrak{t}_{n}\bigoplus {}_{\alpha \in \Phi
_{n}^{+}}sp\{RE_{\alpha },IE_{\alpha }\}=\mathfrak{t}_{n}\bigoplus sp 
\mathcal{V}_{n}
\end{equation*}
where $sp$ denotes the real span. Thus the dimension of $\mathfrak{g}_{n}$
is equal to $n+|$ $\Phi _{n}|.$

From (\ref{adroot}) it follows that 
\begin{equation}
[H,RE_{\alpha }]=-\alpha (H)IE_{\alpha }\text{ and }[H,IE_{\alpha }]=\alpha
(H)RE_{\alpha }.  \label{brtorus}
\end{equation}
It is also well known that 
\begin{equation*}
[RE_{\alpha },IE_{\alpha }]=\frac{-1}{2i}[E_{\alpha },E_{-\alpha }]
\end{equation*}
is a non-zero element of the maximal torus. It should be noted that if $%
\{\alpha _{j}:j\in J\}\subseteq \Phi _{n}$ is a spanning set for $sp\Phi
_{n} $, then $\{RE_{\alpha _{j}},IE_{\alpha _{j}}]:j\in J\}$ spans $%
\mathfrak{t}_{n}$ .

Since $\{E_{\alpha }\}$ is a Weyl basis, we have 
\begin{align}
[RE_{\alpha },RE_{\beta }] &=cRE_{\alpha +\beta }+dRE_{\beta -\alpha },
\label{bracketaction} \\
[RE_{\alpha },IE_{\beta }] &=cIE_{\alpha +\beta }+dIE_{\beta -\alpha } 
\notag \\
[IE_{\alpha },IE_{\beta }] &=-cRE_{\alpha +\beta }+dRE_{\beta -\alpha }, 
\notag
\end{align}
where $RE_{\gamma }$ and $IE_{\gamma }$ should be understood to be the zero
vector if $\gamma $ is not a root and $c=N_{\alpha ,\beta }/2$, $d=N_{\alpha
,-\beta }/2$.

We refer the reader to \cite{Hu}, \cite{Kn} and \cite{Va} for proofs of
these well known facts and further details on the representation theory of
Lie algebras.

\subsection{Annihilating roots}

We call a root, $\alpha ,$ an \textit{annihilating root of }$X\in t_{n}$ if $%
\alpha (X)=0$ and call $\alpha $ a \textit{non-annihilating root of } $X$
otherwise. The \textit{set of annihilating roots of }$X$, 
\begin{equation*}
\Phi _{X}:=\{\alpha \in \Phi :\alpha (X)=0\},
\end{equation*}
is a root subsystem of $\Phi _{n}$. As we will see, these root subsystems
are critical for understanding properties about orbits and orbital measures,
as are the associated root vectors. We will denote by 
\begin{equation}
\mathcal{N}_{X}:=\{RE_{\alpha },IE_{\alpha }:\alpha \notin \Phi
_{X}\}\subseteq \mathcal{V}_{n},  \label{NX}
\end{equation}
the linearly independent subset of $\mathcal{V}_{n}$ consisting of the real
and imaginary parts of the root vectors corresponding to the
non-annihilating roots of $X$. It is known that $\dim O_{X}=|\mathcal{N}
_{X}| $ \cite[VI.4]{MT}. Indeed, the tangent space at $X$ to $O_{X}$ is
spanned by the vectors in $\mathcal{N}_{X}$ and these are linearly
independent (see the proof of Prop.~\ref{key}).

\subsection{Type of an Element}

The torus of $su(n)$, the classical Lie algebra of type $A_{n-1}$ (or type $%
SU(n)$) consists of the diagonal matrices in $su(n)$. After applying a
suitable Weyl conjugate, any $X$ in the torus can be identified with the $n$%
-vector of the real parts of the diagonal elements, 
\begin{equation*}
X=(\underbrace{a_{1},\dots,a_{1}}_{s_{1}},\dots,\underbrace{a_{m},\dots,a_{m}%
} _{s_{m}}),
\end{equation*}
where the $a_{j}\in \mathbb{R}$ are distinct and $\sum_{j=1}^{m}s_{j}a_{j}=0$%
. This means that $ia_{j}$ is an eigenvalue of the $n\times n$ matrix $X$
with multiplicity $s_{j}$. The set of annihilating roots of $X$ is $\Phi
_{X}=$ $\Psi _{1}\cup \cdots \cup \Psi _{m}$ where 
\begin{align*}
\Psi _{1} &=\{e_{i}-e_{j}:1\leq i\neq j\leq s_{1}\}\text{ and} \\
\Psi _{l} &=\{e_{i}-e_{j}:s_{1}+\cdots +s_{l-1}<i\neq j\leq s_{1}+\cdots
+s_{l}\}\text{ for }l>1.
\end{align*}
Following \cite{GHMathZ}, we say that $X$ is type $SU(s_{1})\times \cdots
\times SU(s_{m})$ as this is the Lie type of its set of annihilating roots.

The torus of $so(2n+1)$, the classical Lie algebra of type $B_{n}$, consists
of block diagonal matrices, with $n$ $2\times 2$ blocks of the form $\left[ 
\begin{array}{cc}
0 & b_{j} \\ 
-b_{j} & 0%
\end{array}
\right] $ having $b_{j}\geq 0$, and a $0$ in the final diagonal position. We
identify $X$ in the torus with the $n$-vector $(b_{1},\dots,b_{n})\in 
\mathbb{R }^{+n}$. Up to a Weyl conjugate, $X$ can thus be identified with
the $n$-vector 
\begin{equation}
X=(\underbrace{0,\dots,0}_{J},\underbrace{a_{1},\dots,a_{1}}_{s_{1}},\dots, 
\underbrace{a_{m},\dots,a_{m}}_{s_{m}})  \label{XinBn}
\end{equation}
where the $a_{j}>0$ are distinct. One can see that $0$ is an eigenvalue of
the $(2n+1)\times (2n+1)$ matrix $X$ with multiplicity $2J+1$ and $\pm
ia_{j} $ are eigenvalues with multiplicity $s_{j}$.

The set of annihilating roots $\Phi _{X}=\Psi _{0}\cup \Psi _{1}\cdots \cup
\Psi _{m}$ where 
\begin{align*}
\Psi _{0} &=\{\pm e_{k},\pm e_{i}\pm e_{j}:1\leq i,j,k\leq J,i\neq j\}\text{
and} \\
\Psi _{l} &=\{e_{i}-e_{j}:J+s_{1}+\cdots +s_{l-1}<i\neq j\leq J+s_{1}+\cdots
+s_{l}\}
\end{align*}
for $l=1,\dots,m$. We will say that $X$ is\textit{\ type} 
\begin{equation*}
B_{J}\times SU(s_{1})\times \cdots \times SU(s_{m}),
\end{equation*}
as this is the Lie type of $\Phi _{X}$. Here by $B_{1}$ we mean the root
subsystem $\{\pm e_{1}\}$, while $SU(1),B_{0}$ and $SU(0)$ are empty (and
typically omitted in the description).

Similarly, if $X$ belongs to the torus of the Lie algebra of type $C_{n}$ or 
$D_{n}$ then, up to a Weyl conjugate, $X$ can be identified with the $n$%
-vector 
\begin{equation*}
X=(\underbrace{0,\dots,0}_{J},\underbrace{a_{1},\dots,a_{1}}_{s_{1}},\dots, 
\underbrace{a_{m},\dots,(\pm )a_{m}}_{s_{m}})
\end{equation*}
where the $a_{j}>0$ are distinct. We remark that the minus sign is needed
only in type $D_{n}$ and only if $J=0$. (This is because the Weyl group in
type $D_{n}$ changes only an even number of signs.) Viewing $X$ as an $%
2n\times 2n$ matrix in $sp(n)$ or $so(2n)$, this means that $0$ is an
eigenvalue of $X$ with multiplicity $2J,$ and $\pm ia_{j}$ are eigenvalues
with multiplicity $s_{j}$.

The set of annihilating roots of $X$ can again be written as $\Phi _{X}=
\Psi _{0}\cup \Psi _{1}\cdots \cup \Psi _{m}$. In this case 
\begin{equation*}
\Psi _{0}=\{\pm 2e_{k},\pm e_{i}\pm e_{j}:1\leq i,j,k\leq J,i\neq j\}
\end{equation*}
when the Lie algebra is type $C_{n}$ and 
\begin{equation*}
\Psi _{0}=\{\pm e_{i}\pm e_{j}:1\leq i,j\leq J,i\neq j\}
\end{equation*}
when the Lie algebra is type $D_{n}$. For $l\geq 1$, the $\Psi _{l}$ are as
in type $B_{n}$, except when $X=(a_{1},\dots,a_{1},\dots,a_{m},\dots,-a_{m})$
in $D_{n}$ when 
\begin{equation*}
\Psi _{m}=\{\pm (e_{i}-e_{j}),\pm (e_{i}+e_{n}):n-s_{m}<i\neq j\leq n-1\}.
\end{equation*}
We will say $X$ is \textit{type} 
\begin{equation*}
C_{J}\,\times SU(s_{1})\times \cdots \times SU(s_{m})\text{ or } D_{J}\times
SU(s_{1})\times \cdots \times SU(s_{m})
\end{equation*}
respectively, as these are the Lie types of $\Phi _{X}$. Here $C_{1}$ is the
subsystem $\{\pm 2e_{1}\}$, $C_{2}$ is $\{\pm 2e_{1},\pm 2e_{2},\pm e_{1}\pm
e_{2}\}$, $D_{2}$ is $\{\pm e_{1}\pm e_{2}\}$ (or type $A_{1}\times A_{1})$, 
$D_{3}$ is defined in the obvious way, and $D_{1}$, $D_{0}\,$, $C_{0}$ are
empty (and often omitted).

Note that there are two distinct subsystems (up to Weyl conjugacy) of
annihilating roots of elements of type $SU(n)$ in $D_{n}$.

\begin{definition}
Suppose $X$ is in the torus of the Lie algebra of type $B_{n}$ and is type $%
B_{J}\,\times SU(s_{1})\times \cdots \times SU(s_{m})$. We will say $X$ is 
\textbf{dominant} $B$ \textbf{type} if $2J\geq \max s_{j}$, and is \textbf{%
dominant} $SU$ \textbf{type} otherwise. We define \textbf{dominant} $C$ 
\textbf{and} $D$ \textbf{type} similarly for $X$ in $C_{n}$ or $D_{n}$.
\end{definition}

It was shown in \cite[Thm.~8.2]{GHMathZ} that for each non-zero $X\in 
\mathfrak{g}_{n}$, there is an integer $k(X)$ such that for $k\geq k(X)$, $%
\mu _{X}^{k}\in L^{1}\bigcap L^{2}(\mathfrak{g}_{n})$ (in particular, $%
\mu_{X}^{k}$ is absolutely continuous with respect to Lebesgue measure) and $%
\mu _{X}^{k}$ is purely singular if $k<k(X)$. A formula was given for $k(X)$
depending only on the type of $X$ and the type of the Lie algebra. For
example, if $X$ is dominant $SU$ type in the Lie algebra of type $B_{n}$, $%
C_{n}$ or $D_{n},$ and not of type $SU(n)$ when the Lie algebra is type $%
D_{n}$, then $k(X)=2$. If $X$ is type $B_{n-1}$, ($C_{n-1}$, $D_{n-1}$ or $%
SU(n-1)$) in the Lie algebra of type $B_{n}$ ($C_{n}$, $D_{n}$ or $SU(n)$),
then $k(X)=n$ and this is the maximal choice required for $k(X)$.

\section{Statement of the Main Result}

\subsection{Eligible and Exceptional Tuples}

We introduce the following terminology.

\begin{notation}
If $X$ is of type $SU(s_{1})\times \cdots \times SU(s_{m})$ in the Lie
algebra of type $A_{n}$, put $S_{X}=\max s_{j}$.

If $X$ is type $B_{J}\times SU(s_{1})\times \cdots \times SU(s_{m})$ in the
Lie algebra of type $B_{n}$, put 
\begin{equation*}
S_{X}=%
\begin{cases}
2J & \text{ if }X\text{ is dominant }B\text{ type} \\ 
\max s_{j} & \text{else}%
\end{cases}%
\end{equation*}
Define $S_{X}$ similarly when $X$ belongs to the Lie algebras of type $C_{n}$
or $D_{n}$.
\end{notation}

If $X\in so(2n+1)$ is dominant $B$ type, then the dimension of the largest
eigenspace of the matrix $X$ is $S_{X}+1$, while if $X$ is dominant $SU$
type, then the dimension of the largest eigenspace is $S_{X}$. In all the
other Lie algebras, $S_{X}$ is the dimension of the largest eigenspace when $%
X$ is viewed as a matrix in the appropriate classical matrix algebra.

\begin{definition}
(i) We will say that the $L$-tuple $(X_{1},X_{2},\dots,X_{L})$ of elements
in the torus of a Lie algebra of type $SU(n+1)$ is \textbf{eligible }in $%
\mathfrak{g}_{n}$ if 
\begin{equation*}
\sum_{i=1}^{L}S_{X_{i}}\leq (L-1)(n+1).
\end{equation*}

(ii) We will say that the $L$-tuple $(X_{1},X_{2},\dots,X_{L})$ of elements
in the torus of a Lie algebra of type $B_{n},C_{n}$ or $D_{n}$ is \textbf{%
eligible} in $\mathfrak{g}_{n}$ if 
\begin{equation}
\sum_{i=1}^{L}S_{X_{i}}\leq (L-1)2n.  \label{eligiblecriteria}
\end{equation}
\end{definition}

\begin{definition}
We will say that $(X_{1},X_{2},\dots,X_{L})\in \mathfrak{t}^{L}$ is an 
\textbf{exceptional tuple} if it is any one of the following:

\begin{itemize}
\item $\mathfrak{g}$ is type $SU(2n)$, $L=2$, $n\geq 2$ and $X_{1}$ and $%
X_{2}$ are both of type $SU(n)\times SU(n)$ (i.e., $X_{i}=(\underbrace{
a_{i},\dots,a_{i}}_{n},\underbrace{-a_{i},\dots,-a_{i}}_{n})$);

\item $\mathfrak{g}$ is type $D_{n}$, $L=2$, $X_{1}$ is type $SU(n)$ and $%
X_{2}$ is either type $SU(n)$ or type $SU(n-1)$ (more precisely, type $%
SU(n-1)\times D_{1}$ or $SU(n-1)\times SU(1)$);

\item $\mathfrak{g}$ is type $D_{4}$, $L=2$, $X_{1}$ is type $SU(4)$ and $%
X_{2}$ is either type $SU(2)\times SU(2)$ and $\Phi _{X_{2}}$ is Weyl
conjugate to a subset of $\Phi _{X_{1}}$, or $X_{2}$ is type $SU(2)\times
D_{2}$;

\item $\mathfrak{g}$ is type $D_{4}$, $L=3$ and $X_{1},X_{2},X_{3}$ are all
of type $SU(4)$ with Weyl conjugate sets of annihilators.
\end{itemize}
\end{definition}

\begin{definition}
We will call $(X_{1},X_{2},\dots,X_{L})$ an \textbf{absolutely continuous
tuple} if $\mu _{X_{1}}\ast \mu _{X_{2}}\ast \cdots \ast \mu _{X_{L}}$ is an
absolutely continuous measure.
\end{definition}

Our main result is that other than for the exceptional tuples, eligibility
characterizes absolute continuity of the convolution product. The proof of
this theorem will occupy most of the remainder of the paper. Here is the
formal statement of the theorem.

\subsection{Main Result}

\begin{theorem}
\label{main}Let $\mathfrak{g}_{n}$ be one of the classical, compact,
connected Lie algebras of type $A_{n}$ with $n\geq 1$, $B_{n}$ with $n\geq
2, $ $C_{n}$ with$\ n\geq 3,$ or $D_{n}$ with $n\geq 4$. Assume non-zero $%
X_{i}$, $i=1,2,\dots,L$ for $L\geq 2,$ belong to the torus of $\mathfrak{g}
_{n} $.

\textrm{(i)} Suppose $(X_{1},X_{2},\dots,X_{L})$ is not an exceptional
tuple. The measure, $\mu _{X_{1}}\ast \mu _{X_{2}}\ast \cdots \ast \mu
_{X_{L}},$ is absolutely continuous with respect to Lebesgue measure on $%
\mathfrak{g}_{n}$ if and only if $(X_{1},X_{2},\dots,X_{L})$ is an eligible
tuple.

\textrm{(ii)} If $(X_{1},X_{2},\dots,X_{L})$ is an exceptional tuple, other
than a pair $(X_{1},X_{2})$ of type $(SU(n),SU(n-1))$ \footnote{%
When we say a pair $(X,Y)$ is of type $(\ast ,\ast \ast )$ we mean that $X$
is of type $\ast $ and $Y$ is of type $\ast \ast $.} in a Lie algebra of
type $D_{n}$ with $n\geq 6$, then the measure $\mu _{X_{1}}\ast \mu
_{X_{2}}\ast \cdots \ast \mu _{X_{L}}$ is not absolutely continuous.
\end{theorem}

\begin{remark}
The characterization of absolute continuity in type $A_{n}$ was previously
established by Wright \cite{Wr}. We will include a proof in this paper as
our approach is completely different and requires little additional effort.
\end{remark}

\begin{remark}
(i) We conjecture that a pair of type $(SU(n),SU(n-1))$ in $D_{n}$ with $%
n\geq 6$ also fails to be absolutely continuous.

(ii) Notice that unlike the case for convolutions of the same orbital
measure (\cite[Thm.~8.2]{GHMathZ}), the property of being absolutely
continuous does not depend only upon the type of the annihilating root
systems of the underlying elements, but also, in some cases, upon their Weyl
conjugacy class.
\end{remark}

In proving both absolute continuity and its failure we will rely crucially
upon the following known geometric properties.

The notation $T_{Z}(O_{X})$ will denote the tangent space to $O_{X}$ at $%
Z\in O_{X}$.

\begin{proposition}
\label{key} The measure $\mu _{X_{1}}\ast \mu _{X_{2}}\ast \cdots \ast \mu
_{X_{L}}$ on $\mathfrak{g}_{n}$ is absolutely continuous with respect to
Lebesgue measure if and only if any of the following hold:

\textrm{(i)} $\sum_{i=1}^{L}O_{X_{i}}\subseteq\mathfrak{g}_{n}$ has
non-empty interior;

\textrm{(ii)} $\sum_{i=1}^{L}O_{X_{i}}\subseteq\mathfrak{g}_{n}$ has
positive Lebesgue measure;

\textrm{(iii)} There exists $g_{i}\in G_{n}$ with $g_{1}=Id$, such that 
\begin{equation}
sp\{Ad(g_{i})(\mathcal{N}_{X_{i}}):i=1,\dots,L\}=\mathfrak{g}_{n},
\label{keyiden}
\end{equation}

\textrm{(iv)} There exists $g_{i}\in G_{n}$ with $g_{1}=Id$, such that 
\begin{equation*}
\sum_{i=1}^{L}T_{Ad(g_{i})(X_{i})}(O_{X_{i}})=\mathfrak{g}_{n}.
\end{equation*}
Furthermore, if the identity holds in (iii) or (iv) for one choice of $%
(g_{2},\dots,g_{L})\in G_{n}^{L-1}$, then it holds for all $%
(g_{2},\dots,g_{L})$ in an open dense subset of $G_{n}^{L-1}$ of full
measure.
\end{proposition}

\begin{remark}
We note that (ii) implies that if $\mu _{X_{1}}\ast \mu _{X_{2}}\ast \cdot
\cdot \cdot \ast \mu _{X_{L}}$ is not absolutely continuous, then $\mu
_{X_{1}}\ast \mu _{X_{2}}\ast \cdots \ast \mu _{X_{L}}$ is a purely singular
measure.
\end{remark}

\begin{proof}
This proposition is a compilation of arguments that can be found in \cite%
{GAFA}, \cite{GHMathZ} and \cite{Ra2}. We include a sketch here for the
convenience of the reader. We will show that (iii) and (iv) are equivalent
and then demonstrate the implications (ii)~$\Rightarrow$~(iv)~$\Rightarrow$
absolute continuity and (iv)~$\Rightarrow$~(i). If $\mu _{X_{1}}\ast \mu
_{X_{2}}\ast \cdots \ast \mu _{X_{L}}$ is absolutely continuous or (i)
holds, then (ii) clearly holds so this completes the equivalence.

(iii)~$\Leftrightarrow$~(iv). It is well known (see \cite{GAFA}, \cite[VI.4]%
{MT}) that 
\begin{equation*}
T_{X}(O_{X})=\{[Y,X]:Y\in \mathfrak{g}_{n}\}.
\end{equation*}
Writing $Y=\sum a_{\alpha }RE_{\alpha }+b_{\alpha }IE_{\alpha }+t$ for some $%
t\in \mathfrak{t}_{n}$ and $a_{\alpha },b_{\alpha }$ real, it is easily seen
that $T_{X}(O_{X})=sp\mathcal{N}_{X}$. Further, $T_{Ad(g)X}(O_{x})=Ad(g)
\left( T_{X}(O_{X})\right) =sp\{Ad(g)\mathcal{N}_{x}\}$, proving the
equivalence of (iii) and (iv).

The final comment is an analyticity argument. Assume (iii) holds, for
example, with $g=(Id,g_{2},\dots,g_{L})$. For any $h=(h_{1},h_{2},%
\dots,h_{L}) \in G_{n}^{L}$ , $h_{1}=Id$, consider the collection $Ad(h_{j})Y
$ for $Y\in \mathcal{N}_{X_{j}}$ and $j=1,\dots,L$, as vectors in $\mathbb{R}%
^{\dim \mathfrak{g}_{n}},$ and form the associated matrix $M(h)$. As (iii)
holds with $g$, there is a suitable square submatrix of $M(g)$ with non-zero
determinant. By analyticity of the determinant map, the determinant of the
corresponding square submatrix of $M(h)$ must be non-zero for an open, dense
subset of $h\in G_{n}^{L-1}$ of full measure. The same argument applies to
(iv).

(ii)~$\Rightarrow$~(iv). Consider the addition map $F:O_{X_{1}}\times \cdots
\times O_{X_{L}}\rightarrow \mathfrak{g}_{n}$ given by $F(Y_{1},%
\dots,Y_{L})=\sum_{j=1}^{L}Y_{j}$. The image of $F$ is $%
\sum_{j=1}^{L}O_{X_{j}}$. If the rank of $F$ is not full at any point in its
domain, then Sard's theorem (\cite[p.~286]{La}) implies the measure of the
image of $F$ is zero. Thus the differential of $F$ at some point $%
Y=(Y_{1},\dots,Y_{L}),$ where $Y_{j}=Ad(g_{j})X_{j}$, has full rank. But the
range of the differential of $F$ at $Y$ is $%
\sum_{j=1}^{L}T_{Y_{j}}(O_{X_{j}})$ and hence this sum must be $\mathfrak{g}
_{n}$.

(iv)~$\Rightarrow$~(i). The hypothesis of (iv) guarantees that the map $F$
defined above has full rank at some point $Y$. By the Implicit function
theorem, $F$ is an open map in a neighbourhood of $Y$ and thus $\text{Im }F$
has non-empty interior.

(iv)~$\Rightarrow$~absolute continuity. This is similar again. To see that
the measure $\mu =\mu _{X_{1}}\ast \mu _{X_{2}}\ast \cdots \ast \mu _{X_{L}}$
is absolutely continuous with respect to $m$, we should show that $\mu (E)=0$
whenever $m(E)=0$. Define $f:G_{n}^{L}\rightarrow \mathfrak{\ g}_{n}$ by 
\begin{equation*}
f(g_{1},\dots,g_{L})=F(Ad(g_{1})X_{1},\dots,Ad(g_{L})X_{L}).
\end{equation*}
By definition, $\mu (E)=m_{G_{n}^{L}}(f^{-1}(E))$. By (iv), the differential
of $f$ has full rank at some point. An analyticity argument ensures that
this is true on a subset of $g\in G_{n}^{L}$ of full measure. An application
of the Implicit function theorem shows $f^{-1}(E)$ has $m_{G_{n}^{L}}$ -
measure zero. For more details see \cite[Thm.~2.2]{Ra2}.
\end{proof}

An immediate corollary of this proposition and the main theorem is the
following.

\begin{corollary}
Suppose $(X_{1},X_{2},\dots,X_{L})$ is eligible and not exceptional. Then $%
\sum_{i=1}^{L}O_{X_{i}}$ has non-empty interior. If $(X_{1},X_{2},%
\dots,X_{L}) $ is either not eligible or is exceptional and not type $%
(SU(n),SU(n-1))$ in $D_{n}$, then $\sum_{i=1}^{L}O_{X_{i}}$ has measure zero.
\end{corollary}

There is a sufficient condition for absolute continuity, established by
Wright in \cite{Wr}, that we will use in the proof of the main theorem to
establish the absolute continuity of certain convolution products of orbital
measures in small rank Lie algebras. We state this result below. By the rank
of a subsystem we mean the dimension of the vector space it spans.

\begin{theorem}
\cite[Thm.~1.3]{Wr}\label{WrCriteria} Let $X_{1},\dots,X_{L}$ belong to the
torus of $\mathfrak{g}_{n}$. Assume 
\begin{equation}
(L-1)\left( |\Phi |-|\Psi |\right) -1\geq \sum_{i=1}^{L}\left( |\Phi
_{X_{i}}|-\min_{\sigma \in W}\left\vert \Phi _{X_{i}}\cap \sigma (\Psi
)\right\vert \right)  \label{WrC}
\end{equation}
for all root subsystems $\Psi \subseteq \Phi $ of rank $n-1$ and having the
property that $sp(\Psi )\cap \Phi =\Psi $. Then $\mu _{X_{1}}\ast \cdots
\ast \mu _{X_{L}}$ is absolutely continuous.
\end{theorem}

\section{Tuples That Are Not Absolutely Continuous}

We begin by establishing the necessity of the conditions which give absolute
continuity.

\subsection{Eligibility is a requirement for absolute continuity}

\begin{lemma}
\label{Elig}If $(X_{1},\dots,X_{L})$ is an absolutely continuous $L$-tuple,
then $(X_{1},\dots,X_{L})$ is eligible.
\end{lemma}

\begin{proof}
Suppose the $L$-tuple, $(X_{1},\dots,X_{L})\in \mathfrak{g}_{n}^{L},$ is not
eligible, that is, 
\begin{equation*}
\sum_{i=1}^{L}S_{X_{i}}\geq (L-1)2n+1\text{ (or }(L-1)(n+1)+1\text{ if } 
\mathfrak{g}_{n}\text{ is type }A_{n}.)
\end{equation*}
Let $\alpha _{i}$ be the eigenvalue of $X_{i}$ with greatest multiplicity
(where we view each $X_{i}$ as a complex matrix of the appropriate size
depending on the Lie type of $\mathfrak{g}_{n}$) and let $g_{i}$ belong to
the associated Lie group, $G_{n}$. Let $V_{i}$ be the eigenspace of $%
Ad(g_{i})(X_{i})$ corresponding to the eigenvalue $\alpha _{i}$.

If $\mathfrak{g}_{n}$ is of type $C_{n}$ or $D_{n}$, then $Ad(g_{i})(X_{i})$
are $2n\times 2n$ matrices and $\dim V_{i}=S_{X_{i}}$, so 
\begin{equation*}
\sum_{i=1}^{L}\dim V_{i}=\sum_{i=1}^{L}S_{X_{i}}\geq (L-1)2n+1.
\end{equation*}
We deduce that 
\begin{align*}
&\dim \bigcap\limits_{i=1}^{L}V_{i} \\
=\ &\sum_{i=1}^{L}\dim V_{i}-\left( \dim (V_{1}+V_{2})+\dim ((V_{1}\cap
V_{2})+V_{3})+\cdots +\dim (\bigcap\limits_{i=1}^{L-1}V_{i}+V_{L})\right) \\
\geq\ &(L-1)2n+1-2n(L-1)\geq 1,
\end{align*}
and hence the matrices, $Ad(g_{i})(X_{i}),$ have a common eigenvector, $v$.
As 
\begin{equation*}
\sum_{i=1}^{L}Ad(g_{i})(X_{i})(v)=\sum_{i=1}^{L}\alpha _{i}v,
\end{equation*}
it follows that $\sum_{i}\alpha _{i}$ is an eigenvalue of $%
\sum_{i}Ad(g_{i})(X_{i})$. Since $\sum_{i}Ad(g_{i})(X_{i})$ is an arbitrary
element of $O_{X_{1}}+\cdots +O_{X_{L}}$, one can see that every element of $%
\sum_{i}O_{X_{i}}$ has eigenvalue $\sum_{i}\alpha _{i}$. This is impossible
if $O_{X_{1}}+\cdots +O_{X_{L}}$ has non-empty interior, thus an application
of Prop.~\ref{key}(i) allows us to conclude that $\mu _{X_{1}}\ast \cdots
\ast \mu _{X_{L}}$ is not absolutely continuous.

The argument is similar if $\mathfrak{g}_{n}$ is type $A_{n},$ viewing $%
X_{i} $ as matrices in $su(n+1)$, acting on $\mathbb{R}^{n+1}$.

In the case when $\mathfrak{g}_{n}$ is type $B_{n}$ we require a slight
variation on the argument since every matrix in the Lie algebra $so(2n+1)$
(the model for type $B_{n})$ has $0$ as an eigenvalue. We use the same
notation as above and first observe that if all $X_{i}$ are dominant $B$
type, then all $\alpha _{i}=0$ and $\dim V_{i}=S_{X_{i}}+1$. Thus $%
\sum_{i=1}^{L}\dim V_{i}\geq (L-1)2n+L+1$. Since the vector spaces $V_{i}$
are subspaces of $\mathbb{R}^{2n+1}$, it follows that 
\begin{equation*}
\dim \bigcap\limits_{i=1}^{L}V_{i}\geq (L-1)2n+L+1-(2n+1)(L-1)\geq 2.
\end{equation*}
Consequently, $0$ is an eigenvalue of every element of $O_{X_{1}}+\cdots
+O_{X_{L}}$ of multiplicity at least two. Again, we can conclude that $%
O_{X_{1}}+\cdots +O_{X_{L}}$ has empty interior and therefore $%
(X_{1},\dots,X_{L})$ is not an absolutely continuous tuple.

If, instead, precisely one $X_{i}$ is dominant $SU$ type, with eigenvalue $%
\alpha \neq 0$ of maximum multiplicity, then $\sum \dim V_{i}\geq (L-1)2n+L$%
. This shows that $\dim \bigcap\limits_{i=1}^{L}V_{i}$ has dimension at
least one and hence every element of $O_{X_{1}}+\cdots +O_{X_{L}}$ has $%
\alpha $ as an eigenvalue, again a contradiction if $(X_{1},\dots,X_{L})$ is
an absolutely continuous tuple.

If two or more $X_{i}$ are dominant $SU$ type, then $(X_{1},\dots,X_{L})$ is
automatically eligible.
\end{proof}

\subsection{Exceptional tuples that are not absolutely continuous}

\begin{lemma}
\label{Excep}Suppose $(X_{1},\dots,X_{L})$ is an exceptional tuple and is
not a pair $(X_{1},X_{2})$ of type $(SU(n),SU(n-1))$ in $D_{n}$ where $n\geq
6$. Then $(X_{1},\dots,X_{L})$ is not an absolutely continuous tuple.
\end{lemma}

\begin{proof}
We will need separate arguments for the various exceptional tuples.

(i) Suppose $X_{1}$ and $X_{2}$ are both type $SU(n)$ in the Lie algebra $%
D_{n}$. Observe that 
\begin{equation*}
\dim \left( sp\{Ad(g_{i})(\mathcal{N}_{X_{i}}):i=1,2\}\right) \leq |\mathcal{%
\ N}_{X_{1}}|+|\mathcal{N}_{X_{2}}|.
\end{equation*}
In this case, $|\mathcal{N}_{X_{i}}|=|\Phi _{n}|/2$. As the dimension of the
Lie algebra is $|\Phi _{n}|+n$ it is clearly impossible for $sp\{Ad(g_{i})( 
\mathcal{N}_{X_{i}}):i=1,2\}$ to be the full Lie algebra. Thus Prop.~\ref%
{key}(iii) proves that this pair is not absolutely continuous.

(ii) Suppose $X_{1}$ and $X_{2}$ are of types $SU(n)$ and $SU(n-1)$,
respectively, in $D_{n}$ with $n=4$ or $5$. For this problem, we will use
the fact that a root system of type $SU(4)$ is isomorphic to one of type $%
D_{3}$. We will explain the argument for $n=4$ and leave $n=5$ as an
exercise.

Let $\pi $ be an automorphism of the root system of type $D_{4}$ (an
isomorphism that preserves the Cartan matrix) that maps the annihilating
roots of $X_{1}$ (those of type $SU(4)$) onto a root subsystem of type $%
D_{3} $. This automorphism extends to an automorphism on the torus of $D_{4}$
which maps $X_{1}$ to the element $\pi (X_{1})$ whose set of annihilating
roots is the $D_{3}$ root subsystem, and it maps $X_{2}$ to the element $\pi
(X_{2})$ whose set of annihilating roots is isomorphic to those of $X_{2}$
and hence is type $SU(3)$ (as this is unique up to Lie isomorphism). It
induces a Lie algebra isomorphism that we also call $\pi $. We have $\pi
(O_{X_{j}})=O_{\pi (X_{j})}$ and 
\begin{equation*}
\pi (T_{Ad(g_{j})(X_{j})}(O_{X_{j}}))=T_{Ad(\pi (g_{j}))(\pi
(X_{j}))}(O_{\pi (X_{j})})
\end{equation*}
where if $g_{j}=\exp H_{j}$, then $\pi (g_{j})=\exp \pi (H_{j})$.

The pair $(\pi (X_{1}),\pi (X_{2}))$ is not eligible in $D_{4}$ as $S_{\pi
(X_{1})}=6$ and $S_{\pi (X_{2})}=3$, so by our previous lemma it is not an
absolutely continuous pair. Consequently, Prop.~\ref{key}(iv) implies that 
\begin{equation*}
\dim \left( \sum_{i=1}^{2}T_{Ad(\pi (g_{i}))\pi (X_{i})}(O_{\pi
(X_{i})})\right) <\dim D_{n}
\end{equation*}
for any choices of $g_{1},g_{2}$. But then a similar statement holds for $%
\sum_{i=1}^{2}T_{Ad(g_{j})X_{j}}(O_{X_{j}})$ and thus $(X_{1},X_{2})$ is not
an absolutely continuous pair.

(iii) When $(X_{1},X_{2})$ is a pair of type $(SU(4),$ $SU(2)\times D_{2})$
in $D_{4}$ the arguments are similar. The Lie isomorphism, $\pi ,$ that maps
the subsystem of type $SU(4)$ onto one of type $D_{3}$ must preserve the
type of the root subsystem of type $SU(2)\times D_{2}$. But the pair $(\pi
(X_{1}),\pi (X_{2}))$ is not eligible and hence neither it, nor the original
pair, can be absolutely continuous.

Next, suppose $X_{1}$ is type $SU(4)$ and $X_{2}$ is type $SU(2)\times SU(2)$
in $D_{4}$ with the subsystem, $\Phi _{X_{2}},$ Weyl conjugate to a subset
of the subsystem $\Phi _{X_{1}}$. Since any Weyl conjugate of $X_{2}$
generates the same orbit as $X_{2}$ there is no loss of generality in
assuming $\Phi _{X_{2}}\subseteq \Phi _{X_{1}}$. Consider the same Lie
isomorphism $\pi $ again. Then $\pi (\Phi _{X_{2}})\subseteq\pi (\Phi
_{X_{1}})$ has the same Lie type as $\Phi _{X_{2}}$. But the only subsystems
of type $D_{3}$ that are isomorphic to type $SU(2)\times SU(2)$ are of the
form $\{\pm e_{i}\pm e_{j}\}$ for some $i\neq j,$ and hence are type $D_{2}$%
. Being of type $(D_{3},D_{2})$, the pair $(\pi (X_{1}),\pi (X_{2}))$ is not
eligible and therefore $(X_{1},X_{2})$ is not absolutely continuous.

(iv) Assume $X_{1},X_{2},X_{3}$ are each of type $SU(4)$ in $D_{4}$, with
Weyl conjugate sets of annihilators. As the annihilators are Weyl conjugate,
for each $i=1,2,3$ there exist $h_{i}$ in the Lie group of type $D_{4}$ such
that $Ad(h_{i})(\mathcal{N}_{X_{1}})=\mathcal{N}_{X_{i}}$. Therefore there
exist $g_{i}$ in the group such that 
\begin{equation*}
sp\{Ad(g_{i})(\mathcal{N}_{X_{i}}):i=1,2,3\}=\mathfrak{g}
\end{equation*}
if and only if 
\begin{equation*}
sp\{Ad(g_{i}h_{i})(\mathcal{N}_{X_{1}}):i=1,2,3\}=\mathfrak{g}.
\end{equation*}
But the latter was shown to be impossible in the proof of \cite[Thm.~8.2]%
{GHMathZ}.

(v) The argument is similar if $X_{1}$ and $X_{2}$ are both type $%
SU(n)\times SU(n)$ in the Lie algebra of type $SU(2n)$. In this case, $%
\mathcal{N}_{X_{1}}$ and $\mathcal{N}_{X_{2}}$ are Weyl conjugate and it was
shown in \cite[Prop.~5.1]{GHMathZ} that there is no $g\in SU(2n)$ such that $%
sp\{Ad(g)\mathcal{N}_{X_{1}},\mathcal{N}_{X_{1}}\}=su(2n)$.
\end{proof}

\section{Proving Absolute Continuity - Main Ideas}

\subsection{General Strategy}

Our proof that the eligible, non-exceptional tuples are absolutely
continuous will proceed by induction on the rank of the Lie algebra. The
reduction is based upon the following idea.

\begin{notation}
Suppose $X$ in the torus of the Lie algebra of type $SU(n)$, $B_{n}$, $C_{n}$
or $D_{n}$ is identified (after a suitable Weyl conjugate) with the $n$%
-vector 
\begin{equation*}
(\underbrace{0,\dots,0}_{J},\underbrace{a_{1},\dots,a_{1}}_{s_{1}},\dots, 
\underbrace{a_{m},\dots,(\pm )a_{m}}_{s_{m}})
\end{equation*}
where $s_{1}=\max s_{j}$ and $J=0$ in the case of type $SU(n)$. Define the
element $X^{\prime }\in \mathfrak{t}_{n-1}$ by 
\begin{equation}
X^{\prime }= 
\begin{cases}
(\underbrace{0,\dots,0}_{J-1},\underbrace{a_{1},\dots,a_{1}}_{s_{1}},\dots, 
\underbrace{a_{m},\dots,(\pm )a_{m}}_{s_{m}}) & \text{ if }2J\geq s_{1} \\ 
(\underbrace{0,\dots,0}_{J},\underbrace{a_{1},\dots,a_{1}}_{s_{1}-1},\dots, 
\underbrace{a_{m},\dots,(\pm )a_{m}}_{s_{m}}) & \text{if }2J<s_{1}%
\end{cases}
\label{X'}
\end{equation}
\end{notation}

This means, for example, that if $X$ has type $B_{J}\times SU(s_{1})\times
\cdots \times SU(s_{m})$ where $s_{1}=\max s_{j}$, then $X^{\prime }$ has
type $B_{J-1}\times SU(s_{1})\times \cdots \times SU(s_{m})$ if $X$ is
dominant $B$ type and $X^{\prime }$ has type $B_{J}\times SU(s_{1}-1)\times
\cdots \times SU(s_{m})$ if $X $ is dominant $SU$ type. If $X$ in $SU(n)$
has type $SU(s_{1})\times \cdots \times SU(s_{m})$, then $S_{X^{\prime
}}=S_{X}-1$ if $s_{1}>\max_{j\geq 2}s_{j}$, and $S_{X^{\prime }}=S_{X}$
otherwise. In the latter case $S_{X}\leq n/2$.

We can embed $\mathfrak{t}_{n-1}$ into $\mathfrak{t}_{n}$ by taking the
standard basis vectors $e_{1},\dots,e_{n}$ in $\mathbb{R}^{n}$ (or $%
e_{1}-e_{n+1},\dots,e_{n}-e_{n+1}$ in $\mathbb{R}^{n+1}$ in the case of type 
$SU(n+1)$) as the basis for $\mathfrak{t}_{n}$ and taking the vectors $%
e_{2},\dots,e_{n}$ (resp., $e_{2}-e_{n+1},\dots,e_{n}-e_{n+1})$ as the basis
for $\mathfrak{t}_{n-1}$. This also gives a natural embedding of $\Phi _{n-1}
$ into $\Phi _{n}$ and together these give an embedding of $\mathfrak{g}%
_{n-1}$ into $\mathfrak{g}_{n}$, an embedding of $\mathcal{V}_{n-1}$ into $%
\mathcal{V }_{n}$ and an embedding of $G_{n-1}$ into $G_{n}$. We will also
view $X^{\prime }$ as an element of $\mathfrak{t}_{n}$ in the natural way.

An induction argument will be applicable because of the following lemma.

\begin{lemma}
\label{eligible}If $(X,Y)$ is an eligible pair in $\mathfrak{g}_{n}$ and $%
X,Y $ are not both of type $SU(m)\times SU(m)$ in the Lie algebra of type $%
SU(2m) $, then the reduced pair, $(X^{\prime },Y^{\prime }),$ is eligible in 
$\mathfrak{g}_{n-1}$.
\end{lemma}

\begin{proof}
Case 1: $\mathfrak{g}_{n}$ is type $B_{n}$, $C_{n}$ or $D_{n}$.

Observe that always $S_{X^{\prime }}\leq S_{X}$ since the dimensions of the
eigenspaces of $X^{\prime }$ can only be at most the dimensions of those of $%
X.$

If both $X$ and $X^{\prime }$ are dominant $B,$ $C$ or $D$ type, then $%
S_{X^{\prime }}=S_{X}-2$. If $X^{\prime }$ is dominant $SU$ type, then $%
S_{X^{\prime }}\leq n-1$, regardless of the type of $X$. Finally, if $X$ is
dominant $SU$ type while $X^{\prime }$ is dominant $B$, $C$ or $D$ type,
then $S_{X}=s_{1}>2J=S_{X^{\prime }}\geq s_{1}-1$. Since it is always true
that $J+s_{1}\leq n$, one can check that $s_{1}\leq (2n+1)/3$ and hence $%
S_{X^{\prime }}\leq n-1.$

Thus if either $X$ and $X^{\prime }$ or $Y$ and $Y^{\prime }$ are both
dominant $B,$ $C$ or $D$ type, then 
\begin{equation*}
S_{X^{\prime }}+S_{Y^{\prime }}\leq S_{X}+S_{Y}-2\leq 2(n-1).
\end{equation*}
Otherwise, both $S_{X^{\prime }}$ and $S_{Y^{\prime }}\leq n-1$ and again we
conclude that $S_{X^{\prime }}+S_{Y^{\prime }}\leq 2(n-1)$.

Case 2: $\mathfrak{g}_{n}$ is type $SU(n+1)$.

If either $S_{X^{\prime }}<S_{X}$ or $S_{Y^{\prime }}<S_{Y}$, then $%
S_{X^{\prime }}+S_{Y^{\prime }}\leq S_{X}+S_{Y}-1$ and thus $(X^{\prime
},Y^{\prime })$ is eligible. Otherwise, $S_{X^{\prime }}=S_{X}$ and $%
S_{Y^{\prime }}=S_{Y}$ and in that case $S_{X^{\prime }},S_{Y^{\prime }}\leq
(n+1)/2$. If $n$ is even, then we must have $S_{X^{\prime }},S_{Y^{\prime
}}\leq n/2$ giving $S_{X^{\prime }}+S_{Y^{\prime }}\leq n$. If $n$ is odd,
it is still true that $S_{X^{\prime }}+S_{Y^{\prime }}\leq n$ unless $%
S_{X^{\prime }}=S_{Y^{\prime }}=(n+1)/2$. But that happens only when $X$ and 
$Y$ are both type $SU((n+1)/2)\times SU((n+1)/2)$, which is not permitted.
\end{proof}

\begin{remark}
\label{switchtype}It is easy to see that if $X$ and $X^{\prime }$ are of
opposite dominant types, then $X$ is type $B_{J},$ $(C_{J}$ or $D_{J})\times
SU(s_{1})\times \cdots \times SU(s_{m})$ where $\,1\leq J<\sum s_{i}$. It
follows from \cite[Thm.~8.2]{GHMathZ} that $\mu _{X}^{2}\in L^{2}$.
\end{remark}

We record here a well known fact from elementary linear algebra that is a
consequence of the continuity of the determinant function and will be quite
useful for us.

\begin{lemma}
\label{elem}If $\{v_{1},\dots,v_{n}\}$ is a set of linearly independent
vectors in vector space $V$ and $w_{1},\dots,w_{n}\in V$, then for
sufficiently small $\varepsilon >0$, the collection $\{v_{1}+\varepsilon
w_{1},\dots,v_{n}+\varepsilon w_{n}\}$ is also linearly independent.
\end{lemma}

\begin{notation}
Given $X^{\prime }$ as defined above, let $\mathcal{N}_{X^{\prime
}}=\{RE_{\alpha },IE_{\alpha }:\alpha \notin \Phi _{X^{\prime }}\}$, (as in %
\ref{NX}), but viewed as embedded into $\mathcal{V}_{n}$. Let 
\begin{equation*}
\Omega _{X}=\mathcal{N}_{X}\diagdown \mathcal{N}_{X^{\prime }}.
\end{equation*}
\end{notation}

We will refer to the next result as our general strategy. It will enable us
to establish Prop.~\ref{key}(iii) holds for a given tuple.

\begin{proposition}
\label{general strategy}\textrm{(}General Strategy\/\textrm{)} Let $X_{i}\in 
\mathfrak{t}_{n}$, $i=1,\dots,L$ for $L\geq 2$, and assume $(X_{1}^{\prime
},\dots,X_{L}^{\prime })$ is an absolutely continuous tuple in $\mathfrak{g}%
_{n-1}$. Suppose $\Omega $ is a subset of $\mathcal{V}_{n}\backslash 
\mathcal{V}_{n-1}$ that contains all $\Omega _{X_{i}}$ and has the property
that $ad(H)(\Omega )\subseteq sp\Omega $ whenever $H\in \mathfrak{g}_{n-1}$.
Fix $\Omega _{0}\subseteq \Omega _{X_{L}}$.

Assume there exists $g_{1},\dots g_{L-1}\in G_{n-1}$ and $M\in \mathfrak{g}
_{n}$ such that

\textrm{(i)} $sp\{Ad(g_{i})(\Omega _{X_{i}}),\Omega _{X_{L}}\backslash
\Omega _{0}:i=1,\dots,L-1\}=sp\Omega ;$

\textrm{(ii)} $ad^{k}(M):\mathcal{N}_{X_{L}}\backslash \Omega
_{0}\rightarrow sp\{\Omega ,\mathfrak{g}_{n-1}\}$ for all positive integers $%
k$; and

\textrm{(iii)} The span of the projection of $Ad(\exp sM)(\Omega _{0})$ onto
the orthogonal complement of $sp\{\mathfrak{g}_{n-1},\Omega \}$ in $%
\mathfrak{g} _{n}$ is a surjection for all small $s>0.$

Then $(X_{1},\dots,X_{L})$ is an absolute continuous tuple.
\end{proposition}

\begin{proof}
As ($X_{1}^{\prime },\dots,X_{L}^{\prime })$ is an absolutely continuous
tuple, Prop.~\ref{key}(iii) tells us that 
\begin{equation*}
sp\{Ad(h_{i})(\mathcal{N}_{X_{i}^{\prime }}),\mathcal{N}_{X_{L}^{\prime
}}:i=1,\dots,L-1\}=\mathfrak{g}_{n-1}
\end{equation*}
for a dense set of $(h_{1},\dots,h_{L-1})\in G_{n-1}^{L-1}$. Given $%
\varepsilon >0$, choose such $h_{i}=h_{i}(\varepsilon )\in G_{n-1}$ with $%
\left\Vert Ad(h_{i})-Ad(g_{i})\right\Vert <\varepsilon ,$ where the elements 
$g_{i}\in G_{n-1}$ are the ones given in the hypothesis of the proposition.
(The norm can be taken to be the operator norm.)

Lemma \ref{elem}, together with assumption (i), shows that for sufficiently
small $\varepsilon >0$, 
\begin{align*}
\dim (sp\Omega ) &=\dim \left( sp\{Ad(g_{i})(\Omega _{X_{i}}),\Omega
_{X_{L}}\backslash \Omega _{0}:i=1,\dots,L-1\}\right) \\
&=\dim \left( sp\{Ad(h_{i})(\Omega _{X_{i}}),\Omega _{X_{L}}\backslash
\Omega _{0}:i=1,\dots,L-1\}\right) .
\end{align*}

Since $ad(H)(\Omega )\subseteq sp(\Omega )$ for all $H\in \mathfrak{g}_{n-1}$
and $h_{i}=\exp H_{i}$ for some $H_{i}\in \mathfrak{g}_{n-1}$ we have 
\begin{align*}
Ad(h_{i})(\Omega ) &=Ad(\exp H_{i})(\Omega ) \\
&=\exp (ad(H_{i})(\Omega )\subseteq sp\Omega
\end{align*}
for all $h_{i}\in G_{n-1}$. Thus for sufficiently small $\varepsilon >0$, 
\begin{equation*}
sp\{Ad(h_{i})(\Omega _{X_{i}}),\Omega _{X_{L}}\backslash \Omega
_{0}:L=1,\dots,L-1\}=sp\Omega
\end{equation*}
For such a choice of $\varepsilon $ (hereafter fixed) we have 
\begin{gather*}
sp\{Ad(h_{i})(\mathcal{N}_{X_{i}}),\mathcal{N}_{X_{L}}\backslash \Omega
_{0}:i=1,\dots,L-1\} \\
=sp\{Ad(h_{i})(\mathcal{N}_{X_{i}^{\prime }}),Ad(h_{i})\Omega _{X_{i}}, 
\mathcal{N}_{X_{L}^{\prime }},\Omega _{X_{L}}\backslash \Omega
_{0}\}=sp\{\Omega ,\mathfrak{g}_{n-1}\}.
\end{gather*}

Assumption (ii), and the fact that $\mathcal{N}_{X_{L}}\backslash \Omega
_{0}\subseteq sp\{\Omega ,\mathfrak{g}_{n-1}\},$ implies that for any real
number $s$, $\exp (s\cdot adM)=Ad(\exp sM)$ maps $\mathcal{N}
_{X_{L}}\backslash \Omega _{0}$ to $sp\{\Omega ,\mathfrak{g}_{n-1}\}$.
Moreover, $\left\Vert Id-Ad(\exp sM)\right\Vert \rightarrow 0$ as $%
s\rightarrow 0$, thus similar reasoning to that above shows that for all
small enough $s>0$, 
\begin{align*}
sp\{\Omega ,\mathfrak{g}_{n-1}\} &=sp\{Ad(h_{i})(\mathcal{N}_{X_{i}}), 
\mathcal{N}_{X_{L}}\backslash \Omega _{0}:i=1,\dots,L\} \\
&=sp\{Ad(h_{i})(\mathcal{N}_{X_{i}}),(Ad(\exp sM))(\mathcal{N}
_{X_{L}}\backslash \Omega _{0}):i=1,\dots,L-1\}.
\end{align*}

Combined with assumption (iii), this proves that for sufficiently small $%
s>0, $ 
\begin{equation*}
sp\{Ad(h_{i})(\mathcal{N}_{X_{i}}),Ad(\exp sM)(\mathcal{N}
_{X_{L}}):i=1,\dots,L-1\}=\mathfrak{g}_{n}.
\end{equation*}
Another application of Prop.~\ref{key}(iii) shows that $\mu _{X_{1}}\ast
\cdots \ast \mu _{X_{L}}$ is absolutely continuous.
\end{proof}

We will occasionally make use of the following specific application of the
elementary linear algebra property in order to verify the hypothesis of the
general strategy.

\begin{lemma}
\label{genstrategyb}Suppose $\Omega $ is a subset of $\mathcal{V}
_{n}\backslash \mathcal{V}_{n-1}$ that contains both $\Omega _{X}$ and $%
\Omega _{Y},$ and has the property that $ad(H)(\Omega )\subseteq sp\Omega $
whenever $H\in \mathfrak{g}_{n-1}$. Fix $\Omega _{0}\subseteq \Omega _{X}$.
Assume $\Omega _{1}\subseteq (\Omega _{Y}\cap \Omega _{X})\diagdown \Omega
_{0}$ and the vectors in $\{adH(\Omega _{1}),$ $\Omega _{Y}\diagdown \Omega
_{1},$ $\Omega _{X}\diagdown \Omega _{0}\}$ span $\Omega $ for some $H\in 
\mathfrak{g}_{n-1}$. Then for sufficiently small $t>0$, 
\begin{equation*}
sp\{Ad(\exp tH)(\Omega _{Y}),\Omega _{X}\backslash \Omega _{0}\}=sp\Omega .
\end{equation*}
\end{lemma}

\begin{proof}
The arguments are similar to that of the general strategy. Since $\left\Vert
ad(H)-\frac{1}{t}(Ad(\exp tH)-Id)\right\Vert $ and $\left\Vert Id-Ad(\exp
tH)\right\Vert $ both tend to $0$ as $t\rightarrow 0$, and $ad^{k}(H)(\Omega
)\subseteq sp\Omega $ for all $k$, the same argument as used above shows
that 
\begin{equation*}
sp\{(Ad(\exp tH)-Id)\left( \Omega _{1}\right) ,Ad(\exp tH)(\Omega
_{Y}\backslash \Omega _{1}),\Omega _{X}\backslash \Omega _{0}\}=sp\Omega .
\end{equation*}
But since $\Omega _{1}\subseteq \Omega _{X}\backslash \Omega _{0}$, we can
replace $(Ad(\exp tH)-Id)\left( \Omega _{1}\right) $ in the span on the left
hand side by $Ad(\exp tH)\left( \Omega _{1}\right) $. Hence 
\begin{equation*}
sp\{Ad(\exp tH)\left( \Omega _{Y}\right) ,\Omega _{X}\backslash \Omega
_{0}\}=sp\Omega \text{.}
\end{equation*}
\end{proof}

\subsection{Applying the General strategy with $L=2$}

The following proposition, the `induction step', is the most important
ingredient in the proof of the main theorem.

We continue to use the notation $\Omega _{X}=\mathcal{N}_{X}\diagdown 
\mathcal{N}_{X^{\prime }}$, where $X^{\prime }$ was defined in (\ref{X'}).

\begin{proposition}
\label{indstep}Suppose $(X,Y)$ is an eligible pair in $\mathfrak{g}_{n}$
other than $X,Y$ both of type $SU(n)$ in $D_{n}$ or type $SU(n/2)\times
SU(n/2)$ in $SU(n)$. Assume also that the reduced pair, $(X^{\prime
},Y^{\prime }),$ is an absolutely continuous pair in $\mathfrak{g}_{n-1}$.
Then $(X,Y)$ is an absolutely continuous pair in $\mathfrak{g}_{n}.$
\end{proposition}

\begin{proof}
The main task of the proof is to show that any eligible pair, other than one
of the two exceptional pairs mentioned, satisfy properties (i) - (iii) of
the general strategy, Prop.~\ref{general strategy}.

\textbf{Part I: $g_{n}$ is type $B_{n}$, $C_{n}$ or $D_{n}$.}

The proof is divided into three cases depending on the dominant types of $X$
and $Y$.

Case 1: Neither $X$\ nor $Y$ are of dominant $SU$ type.

With the notation as before, we have $S_{X}=2J$ and $S_{Y}=2K$ (meaning $X$
is dominant $B_{J}$ ($C_{J}$ or $D_{J}$) type and $Y$ is dominant $B_{K}$ ($%
C_{K}$ or $D_{K}$) type). Applying a Weyl conjugate, if necessary, we can
assume without loss of generality that 
\begin{equation*}
\Omega _{X}=\{FEe_{1}\pm e_{j}:J<j\leq n,\text{ }F=R,I\}
\end{equation*}
and similarly 
\begin{equation*}
\Omega _{Y}=\{FEe_{1}\pm e_{j}:K<j\leq n,F=R,I\}.
\end{equation*}

Case 1(a): $g_{n}$ is type $D_{n}$.

Recall that $\mathcal{V}_{n}$ is set of all real and imaginary parts of the
chosen Weyl basis of root vectors of $\mathfrak{g}_{n}$. Put 
\begin{equation*}
\Omega =\mathcal{V}_{n}\diagdown \mathcal{V}_{n-1}=\{FEe_{1}\pm
e_{j}:j=2,\dots,n,F=R,I\}
\end{equation*}
and 
\begin{equation*}
\Omega _{0}=\{REe_{1}+e_{n},IEe_{1}+e_{n}\}.
\end{equation*}

If $H\in \mathfrak{g}_{n-1}$, then $H$ is a linear combination of a torus
element of $\mathfrak{g}_{n-1}$ and the vectors $REe_{i}\pm e_{j}$, $%
IEe_{i}\pm e_{j}$ with $2\leq i<j\leq n$. It follows easily from (\ref%
{bracketaction}) that $ad(H)(\Omega )\subseteq sp\Omega $.

Take $g\in G_{n-1}$ to be the Weyl conjugate that permutes the letters $1+j$
and $K+j$ for $j=1,\dots,J-1$. This is well defined and leaves the letter $n$
unchanged as the eligibility condition ensures $J+K-1\leq n-1$.
Consequently, $Ad(g)(FEe_{1}\pm e_{K+j})=FEe_{1}\pm e_{1+j}$ for $%
j=1,\dots,J-1,$ and all other vectors in $\Omega $ are fixed, including $%
FEe_{1}\pm e_{n}$. Thus 
\begin{equation*}
\{Ad(g)(\Omega _{Y}),\text{ }\Omega _{X}\diagdown \Omega _{0}\}=\{FEe_{1}\pm
e_{k}:k=2,\dots,n\}
\end{equation*}
proving that (i) of the general strategy, Prop~\ref{general strategy} (with $%
L=2$) is satisfied.

Let $M=REe_{1}+e_{n}\in \mathfrak{g}_{n}$. Applying (\ref{bracketaction})
again, we see that if $H=FEe_{1}\pm e_{j}$ for some $j<n$, then $%
ad(M)(H)=cFEe_{j}\mp e_{n}\in \mathfrak{g}_{n-1}$ for a non-zero constant $c$
depending on $j,n$ and $F$. If $H=FEe_{i}\pm e_{n}$, then $%
ad(M)(H)=cFEe_{1}\mp e_{i}\in sp(\Omega \diagdown \Omega _{0})$. Finally,
note that $ad(M)(H)=0$ if $H=FEe_{i}\pm e_{j}$ for $1<i,j<n$ or $%
H=FEe_{1}-e_{n}$. This proves $ad^{k}(M):\mathcal{N}_{X}\backslash \Omega
_{0}\rightarrow sp\{\Omega ,\mathfrak{g}_{n-1}\}$ for all positive integers $%
k,$ so that property (ii) of the general strategy is satisfied.

As $sp\{\Omega ,\mathfrak{g}_{n-1}\}$ is of co-dimension one, its orthogonal
complement is spanned by the projection onto any element in the complement
of $sp\{\Omega ,\mathfrak{g}_{n-1}\}$. The torus element, 
\begin{equation*}
ad(M)(IEe_{1}+e_{n})=[REe_{1}+e_{n},IEe_{1}+e_{n}]:=t_{1}
\end{equation*}
is such an element. Since 
\begin{equation*}
ad(M)(t_{1})=[REe_{1}+e_{n},t_{1}]=cIEe_{1}+e_{n}
\end{equation*}
for some $c\neq 0$ (see \ref{brtorus}), it follows that 
\begin{equation*}
Ad(\exp sM)(IEe_{1}+e_{n})=a(s)IEe_{1}+e_{n}+sb(s)t_{1}
\end{equation*}
where $a(s),b(s)\rightarrow 1$ as $s\rightarrow 0$. Therefore hypothesis
(iii) of Prop.~\ref{general strategy} is also fulfilled with any $s>0$.
Applying that proposition, we conclude that $\mu _{X}\ast \mu _{Y}$ is
absolutely continuous.

Case 1(b): $g_{n}$\ is type $B_{n}$.

Again, we will apply the general strategy, but here with 
\begin{equation*}
\Omega =\mathcal{V}_{n}\backslash \mathcal{V}_{n-1}=\{FEe_{1}\pm e_{j}\text{
, }FEe_{1}:j=2,\dots,n,F=R,I\}
\end{equation*}
and 
\begin{equation*}
\Omega _{0}=\{REe_{1}+e_{n},IEe_{1}+e_{n}\}.
\end{equation*}

The fact that $ad(H)(\Omega )\subseteq \Omega $ whenever $H\in \mathfrak{g}
_{n-1}$ follows easily from properties of the roots, as with the case $D_{n}$%
.

For $t>0$, let $g_{t}=(\exp tREe_{n})g$ where $g\in G_{n-1}$ corresponds to
the Weyl conjugate that permutes the letters $1+j$ and $K+j$ for $%
j=1,\dots,J-1 $ as in the previous case. Since $REe_{n}\in \mathfrak{g}_{n-1}
$, $g_{t}\in G_{n-1}$. Observe that 
\begin{equation*}
\left[ REe_{n},FEe_{1}\right] =cFEe_{1}+e_{n}\;+c^{\prime }FEe_{1}-e_{n}
\end{equation*}
and 
\begin{equation*}
\left[ REe_{n},FEe_{1}\pm e_{j}\right] = 
\begin{cases}
c^{(\pm )}FEe_{1} & \text{if }j=n \\ 
0 & \text{else}%
\end{cases}%
\end{equation*}
with $c,c^{\prime }$ and $c^{(\pm )}$ non-zero constants. In particular,
this implies 
\begin{equation*}
Ad(\exp tREe_{n})(FEe_{1}\pm e_{j})=FEe_{1}\pm e_{j}\text{ for }j\neq n.
\end{equation*}
Since $Ad(g)(FEe_{1}\pm e_{K+j})=FEe_{1}\pm e_{1+j}$ for $j=1,\dots,J-1$ and
the eligibility condition ensures $Ad(g)$ fixes $FEe_{1}\pm e_{n}$, it
follows that for $j=1,\dots,J-1$ we have 
\begin{align*}
Ad(g_{t})(FEe_{1}\pm e_{K+j}) &=Ad(\exp tREe_{n})(FEe_{1}\pm e_{1+j}) \\
&=FEe_{1}\pm e_{1+j}\text{ }
\end{align*}
and 
\begin{align*}
Ad(g_{t})(FEe_{1}\pm e_{n}) &=Ad(\exp tREe_{n})(FEe_{1}\pm e_{n}) \\
&=a(t)FEe_{1}\pm e_{n}\;+tb(t)FEe_{1}\;+t^{2}c(t)FEe_{1}\mp e_{n}
\end{align*}
where $a(t)\rightarrow 1$ as $t\rightarrow 0$, and $b(t)$ and $c(t)$
converge to non-zero scalars\footnote{$a(t),b(t),c(t)$ depend on $F$ and the
choice of $\pm $, as well as $t$. From here on we will omit noting this
dependence, unless it is important.}. All other choices of $FEe_{1}\pm e_{j}$
are fixed by $Ad(g_{t})$. Hence 
\begin{align*}
&sp\{FEe_{1}-e_{n},Ad(g_{t})(FEe_{1}\pm e_{n}):F=R,I\} \\
=\ &sp\{FEe_{1}-e_{n},FEe_{1}+e_{n}\;+tb^{\prime
}(t)FEe_{1},FEe_{1}\;+tc^{\prime }(t)FEe_{1}+e_{n}:F=R,I\}
\end{align*}
where $b^{\prime }(t)$ and $c^{\prime }(t)$ converge to non-zero limits as $%
t\rightarrow 0$. Since 
\begin{equation*}
\{FEe_{1}\pm e_{n},FEe_{1}:F=R,I\}
\end{equation*}
is a set of six linearly independent vectors, so too is the collection 
\begin{equation*}
\{FEe_{1}-e_{n},FEe_{1}+e_{n}\;+tb^{\prime }(t)FEe_{1},\
FEe_{1}\;+tc^{\prime }(t)FEe_{1}+e_{n}:F=R,I\}
\end{equation*}
for sufficiently small $t$, and therefore they span the same space. Because $%
\Omega _{X}\backslash \Omega _{0}$ contains $FEe_{1}-e_{n}$, it follows that 
\begin{align*}
&sp\{Ad(g_{t})(\Omega _{Y}),\Omega _{X}\backslash \Omega _{0}\} \\
=\ &sp\{Ad(g_{t})(FEe_{1}\pm e_{k}),FEe_{1}\pm
e_{j},FEe_{1}-e_{n}:k>K,J<j<n,F=R,I\} \\
=\ &sp\{FEe_{1}\pm e_{j},FEe_{1}\pm e_{n}\text{, }FEe_{1}:j\leq
n,F=R,I\}=sp\Omega.
\end{align*}

Again, put $M=REe_{1}+e_{n}\in \mathfrak{g}_{n}$. As with type $D_{n}$, $%
ad^{k}(M)(FEe_{1}\pm e_{j})\in sp\{\mathfrak{g}_{n-1},\Omega \}$ for all $k$
and $j<n$, and $ad(M)(FEe_{1}-e_{n})=0$. Furthermore, $ad(M)(FEe_{j})=0$ if $%
j\neq 1,n,$ $ad(M)(FEe_{n})=cFEe_{1}$ and $ad(M)(FEe_{1})=cFEe_{n}$, so
property (ii) of the general strategy holds. As in the first case, $sp\{ 
\mathfrak{g}_{n-1},\Omega \}$ is of co-dimension one in $\mathfrak{g}_{n},$
and just as in type $D_{n}$ property (iii) holds, so we deduce the absolute
continuity of $\mu _{X}\ast \mu _{Y}$ by appealing to Prop.~\ref{general
strategy}.

Case 1(c): $g_{n}$\ is type $C_{n}$.

Here we will use a variant on the general strategy. As with type $D_{n}$ we
begin with 
\begin{equation*}
\Omega =\{FEe_{1}\pm e_{j}:j=2,\dots,n,F=R,I\}
\end{equation*}
and $g$ the Weyl conjugate permuting the letters $1+j$ and $K+j$ for $%
j=1,\dots,J-1$. Take 
\begin{equation*}
\Omega _{0}=\{FEe_{1}\pm e_{n}:F=R,I\}.
\end{equation*}
The eligibility condition gives that $sp\{Ad(g)(\Omega _{Y}),\Omega
_{X}\backslash \Omega _{0}\}=sp\Omega $.

As with type $D_{n}$, $ad(FEe_{i}\pm e_{j})(\Omega )\subseteq sp\{\Omega , 
\mathfrak{g}_{n-1}\}$ for all $1<i<j\leq $ $n$ and similarly, $%
ad(FE(2e_{j}))(\Omega )\subseteq \Omega $ for $j>1$, so $ad(H)(\Omega
)\subseteq sp\{\Omega ,\mathfrak{g}_{n-1}\}$ whenever $H\in \mathfrak{g}
_{n-1}$. Thus, as in the proof of the general strategy, upon applying the
induction assumption we can deduce there is some $h\in G_{n-1}$ such that 
\begin{equation}
sp\{Ad(h)(\mathcal{N}_{Y}),\mathcal{N}_{X}\backslash \Omega
_{0}\}=sp\{\Omega ,\mathfrak{g}_{n-1}\}.  \label{property1}
\end{equation}

Once again, we will put $M=REe_{1}+e_{n}\in \mathfrak{g}_{n}$. As with the
types $B_{n}$ and $D_{n}$, standard facts about roots show that $ad(M)(H)\in
sp\{\Omega ,\mathfrak{g}_{n-1}\}$ for all $H\in \mathcal{N}_{X}\backslash
\Omega _{0}$. In fact, for all $k\geq 1$, $ad^{k}(M)(H)\in sp\{\Omega , 
\mathfrak{g}_{n-1}\}$ for all $H\in \mathcal{N}_{X}\backslash \Omega _{0}$ 
\textit{except for }$H=FE(2e_{n})$ as $ad^{k}(M)(FE(2e_{n}))$ has a
component in $FE(2e_{1})$. (Recall that $FE(2e_{n})\in \mathcal{N}_{X}$
since the only roots $2e_{j}\in \Phi _{X}$ are those with $j\leq J$.) It is
because of this exception that we cannot appeal directly to the general
strategy.

Another difference between this set up and the situation for types $B_{n}$
and $D_{n}$ is that here $sp\{\Omega ,\mathfrak{g}_{n-1}\}$ has co-dimension
three, its orthogonal complement being spanned by $RE(2e_{1}),IE(2e_{1})$
and the projection onto the torus element $[REe_{1}+e_{n},IEe_{1}+e_{n}]$.
That will also complicate matters.

Let $\Lambda $ be the subspace spanned by the torus of $\mathfrak{g}_{n-1}$
and the vectors $RE_{\beta }$ and $IE_{\beta }$ where $\beta $ ranges over
all the positive roots except \thinspace $2e_{1},2e_{n}$, 
\begin{equation*}
\Lambda :=sp\{\Omega ,\mathfrak{g}_{n-1}\}\ominus
sp\{RE(2e_{n}),IE(2e_{n})\}.
\end{equation*}
Let $\mathcal{P}$ be the orthogonal projection onto $\Lambda $. Since $%
\mathcal{N}_{X}\diagdown \{\Omega _{0},FE(2e_{n})\}\subseteq \Lambda $,
Property (\ref{property1}) \ implies that 
\begin{equation*}
sp\{\mathcal{P}(Ad(h)(\mathcal{N}_{Y})),\mathcal{N}_{X}\diagdown \{\Omega
_{0},RE(2e_{n}),IE(2e_{n})\}\}=\Lambda .
\end{equation*}
Choose $Y_{\beta }^{F},Y_{j}\in Ad(h)\mathcal{N}_{Y}$ and $X_{\beta
}^{F},X_{j}\in \mathcal{N}_{X}\diagdown \{\Omega
_{0},RE(2e_{n}),IE(2e_{n})\} $ such that

(i) $Y_{\beta }^{F}+X_{\beta }^{F}=FE_{\beta }+W_{\beta }^{F}$ where $%
W_{\beta }^{F}\in sp\{RE(2e_{n}),IE(2e_{n})\}$, $F=R,I$ and $\beta $ ranges
over all roots except $2e_{1},2e_{n},$ and

(ii) $Y_{j}+X_{j}=t_{j}+W_{j}$ where $j=2,\dots,n$, $\{t_{2},\dots,t_{n}\}$
is a basis for $\mathfrak{t}_{n-1}$ and $W_{j}\in sp\{FE(2e_{n})\}$.

Note that if we put $t_{1}=[REe_{1}+e_{n},IEe_{1}+e_{n}]$, then $%
\{t_{1},\dots,t_{n}\}$ is a basis for $\mathfrak{t}_{n}$.

This collection of vectors $\{Y_{\beta }^{F}+X_{\beta }^{F},Y_{j}+X_{j}\}$
is linearly independent and hence for small enough $s>0$, so is also the set 
\begin{equation*}
\{Y_{\beta }^{F}+Ad(\exp sM)(X_{\beta }^{F}),Y_{j}+Ad(\exp sM)(X_{j}):\beta
\neq 2e_{1},2e_{n},\ j=2,\dots,n,F=R,I\}.
\end{equation*}
Observe that 
\begin{align*}
Y_{\beta }^{F}+Ad(\exp sM)(X_{\beta }^{F}) &=Y_{\beta }^{F}+X_{\beta
}^{F}+(Ad(\exp sM)-Id)(X_{\beta }^{F}) \\
&=FE_{\beta }+W_{\beta }^{F}+sQ_{\beta }^{F}(s),
\end{align*}
where the vector $Q_{\beta }^{F}(s)$ depends on $s,$ but has bounded norm.
The projection of $Q_{\beta }^{F}(s)$ onto $sp\{RE2e_{1},IE2e_{1}\}$ is zero
since $Ad(\exp sM)$ maps $\mathcal{N}_{X}\diagdown \{\Omega
_{0},FE(2e_{n}):F=R,I\}$ into $\mathfrak{g}_{n}\ominus
sp\{RE2e_{1},IE2e_{1}\}$. Also, it is clear from the definitions that for $%
\beta \neq e_{1}-e_{n}$, the projection of $FE_{\beta }+W_{\beta }^{F}$ onto 
$sp\{REe_{1}-e_{n},IEe_{1}-e_{n}\}$ is zero. Similar statements can be made
for $Y_{j}+Ad(\exp sM)(X_{j})$.

\textit{Claim:} The collection of vectors, $Y_{\beta }^{F}+Ad(\exp
sM)(X_{\beta }^{F})$, $Y_{j}+Ad(\exp sM)(X_{j})$ over all positive roots $%
\beta \neq 2e_{1},2e_{n}$, $F=R,I$, and $j=2,\dots,n,$ together with the
four vectors $Ad(\exp sM)(FE(2e_{n}))$, $Ad(\exp sM)(FEe_{1}-e_{n})$ for $%
F=R,I$, are linearly independent.

To prove this we first observe that 
\begin{align*}
[REe_{1}+e_{n},FE(2e_{1})] &=c_{1}FEe_{1}-e_{n} \\
[REe_{1}+e_{n},FE(2e_{n})] &=c_{2}FEe_{n}-e_{1}, \\
[REe_{1}+e_{n},FEe_{1}-e_{n}] &=c_{3}FE(2e_{1})+c_{4}FE(2e_{n})
\end{align*}
where $c_{j}\neq 0$. Thus 
\begin{equation}
Ad(\exp sM)(FEe_{1}-e_{n})=a_{s}^{F}FEe_{1}-e_{n}\;+sb_{s}^{F}FE(2e_{1})
+sc_{s}^{F}FE(2e_{n})  \label{four}
\end{equation}
and 
\begin{equation}
Ad(\exp sM)(FE(2e_{n}))=sb_{s}^{\prime F}FEe_{1}-e_{n}+s^{2}c_{s}^{\prime
F}FE(2e_{1})\;+a_{s}^{\prime F}FE(2e_{n})  \label{four'}
\end{equation}
where the coefficients, $a_{s}^{F},a_{s}^{\prime F},b_{s}^{F},b_{s}^{\prime
F},c_{s}^{F},c_{s}^{\prime F},$ converge to non-zero constants as $%
s\rightarrow 0$.

The vectors listed in (\ref{four}) and (\ref{four'}), as well as those in $%
sp\{\mathfrak{g}_{n-1},\Omega \},$ belong to $\mathfrak{g}_{n}\ominus
sp\{t_{1}\}$. We view them as vectors in $\mathbb{R}^{d}$ with $d=\dim 
\mathfrak{g}_{n}-1$, whose coordinates are given by the basis for $\mathfrak{%
\ g}_{n}\ominus sp\{t_{1}\}$ consisting of the torus elements, $%
\{t_{2},\dots,t_{n}\}$, together with the real and imaginary parts of the
Weyl basis $\{E_{\alpha }\}$, taking as the final six positions the basis
vectors $FEe_{1}-e_{n}$, $FE(2e_{n})$ and $FE(2e_{1})$, $F=R,I$.

With this understanding, consider the square matrix whose rows are given by
the vectors $Y_{j}+Ad(\exp sM)(X_{j})$ for $j=2,\dots,n$; followed by the
vectors $Y_{\beta }^{F}+Ad(\exp sM)(X_{\beta }^{F})$, $\beta \neq
2e_{1},2e_{n}$, ordered consistently to above so that the final two come
from $\beta =e_{1}-e_{n}$; and then finally the four vectors $Ad(\exp
sM)(FE(2e_{n}))$ and $Ad(\exp sM)(FEe_{1}-e_{n})$ (for a small, but fixed,
choice of $s$).

The calculations above show that this matrix, denoted $A=(A_{ij}),$ has the
form 
\begin{equation*}
A=\left[ 
\begin{array}{cccc}
\left[ I_{d-6}+O(s)\right] _{(d-6)\times (d-6)} & \left[ O(s)\right]
_{(d-6)\times 2} & [\ast ]_{(d-6)\times 2} & [0]_{(d-6)\times 2} \\ 
&  &  &  \\ 
\left[ O(s)\right] _{2\times (d-6)} & \left[ I_{2}+O(s)\right] _{2\times 2}
& \left[ 
\begin{array}{cc}
\ast & \ast \\ 
\ast & \ast%
\end{array}
\right] & \left[ 
\begin{array}{cc}
0 & 0 \\ 
0 & 0%
\end{array}
\right] \\ 
&  &  &  \\ 
\lbrack 0]_{2\times (d-6)} & \left[ 
\begin{array}{cc}
sb_{s}^{\prime R} & 0 \\ 
0 & sb_{s}^{\prime I}%
\end{array}
\right] & \left[ 
\begin{array}{cc}
a_{s}^{\prime R} & 0 \\ 
0 & a_{s}^{\prime I}%
\end{array}
\right] & \left[ 
\begin{array}{cc}
O(s^{2}) & 0 \\ 
0 & O(s^{2})%
\end{array}
\right] \\ 
&  &  &  \\ 
\lbrack 0]_{2\times (d-6)} & \left[ 
\begin{array}{cc}
a_{s}^{R} & 0 \\ 
0 & a_{s}^{I}%
\end{array}
\right] & \left[ 
\begin{array}{cc}
sc_{s}^{R} & 0 \\ 
0 & sc_{s}^{I}%
\end{array}
\right] & \left[ 
\begin{array}{cc}
sb_{s}^{R} & 0 \\ 
0 & sb_{s}^{I}%
\end{array}
\right]%
\end{array}
\right]
\end{equation*}
where $I_{m}$ denotes the $m\times m$ identity matrix, $O(s^{k})$ means
terms dominated by $Cs^{k}$ for some constant $C$ independent of $s$ and $%
\ast $ denotes terms that may depend on $s$, but are bounded independently
of $s$.

We estimate the determinant of this matrix using the Leibniz formula: Since $%
|A_{11}A_{22}\cdots A_{dd}|\geq C_{0}s^{2}$ for some $C_{0}>0$ and all the
other products \linebreak $A_{1\sigma (1)}A_{2\sigma (2)}\cdots A_{d\sigma
(d)}$, where $\sigma $ is a permutation of $\{1,\dots,d\}$, are dominated in
absolute value by $C_{1}s^{3}$, the determinant is non-zero for sufficiently
small $s>0$. This completes the proof of the claim.

As there are the appropriate number of vectors, these vectors form a basis
for $\mathfrak{g}_{n}\ominus sp\{t_{1}\}$. Recall that $X_{\beta }^{F}$, $%
FE(2e_{n})$ and $FEe_{1}-e_{n}$ all belong to $\mathcal{N}_{X}\diagdown
\{FEe_{1}+e_{n}:F=R,I\}$, hence 
\begin{equation*}
sp\{Ad(h)\mathcal{N}_{Y},Ad(\exp sM)(\mathcal{N}_{X}\diagdown
\{FEe_{1}+e_{n}\})\}=\mathfrak{g}_{n}\ominus sp\{t_{1}\}.
\end{equation*}

Finally, our familiar calculation shows 
\begin{equation*}
Ad(\exp sM)(IEe_{1}+e_{n})=a_{s}IEe_{1}+e_{n}+sb_{s}t_{1}
\end{equation*}
where $b_{s}$ converges to a non-zero constant. It follows that for small
enough $s$, 
\begin{equation*}
sp\{Ad(h)(\mathcal{N}_{Y}),Ad(\exp sM)(\mathcal{N}_{X})\}=\mathfrak{g}_{n},
\end{equation*}
as we desired to show.

Case 2: Both $X$\ and $Y$\ are dominant $SU$ type.

First, assume the Lie algebra is type $B_{n}$ or $C_{n}$. According to \cite[
Thm.~8.2]{GHMathZ} both $\mu _{X}^{2}$ and $\mu _{Y}^{2}$ belong to $L^{2}$.
Applying Holder's inequality we see that $\mu _{X}\ast \mu _{Y}\in L^{2}$.
Being compactly supported, it follows that $\mu _{X}\ast \mu _{Y}$ is in $%
L^{1}$, and hence is a measure that is absolutely continuous with respect to
Lebesgue measure. (Note that the same argument applies to the Lie algebra of
type $D_{n}$ unless one of $X$ or $Y$ is of type $SU(n)$.)

However, we prefer to give an argument that is independent of \cite{GHMathZ}
as the techniques will then have more general application and such an
argument will be needed in the case of type $D_{n}$, in any case. For this,
in the case of type $B_{n}$, put 
\begin{align*}
\Omega &=\{FEe_{1},FEe_{1}\pm e_{j}:j\geq 2,F=R,I\}\text{ and} \\
\text{ }\Omega _{0} &=\{REe_{1},IEe_{1}\}
\end{align*}
while in the case of type $C_{n}$, put 
\begin{align*}
\Omega &=\{FE(2e_{1}),FEe_{1}\pm e_{j}:j\geq 2,F=R,I\}\text{ and } \\
\Omega _{0} &=\{RE(2e_{1}),IE(2e_{1})\}.
\end{align*}
In either case $ad(H)(\Omega )\subseteq sp\Omega $ for all $H\in \mathfrak{g}
_{n-1}$.

As $X,Y$ are dominant $SU$ type, both $\Omega _{X}$ and $\Omega _{Y}$
contain $FE(2)e_{1}$ and all the roots $FEe_{1}+e_{j}$, $j\geq 2$. If $g\in
G_{n-1}$ is the Weyl conjugate that changes the signs of the letters $%
2,\dots,n $, then $\{Ad(g)(\Omega _{X}),\Omega _{Y}\diagdown \Omega
_{0}\}=\Omega $. Now take $M=RE(2)e_{1}$ and apply the general strategy.

The arguments are similar when the Lie algebra is type $D_{n}$. Let 
\begin{equation*}
\Omega =\{FEe_{1}\pm e_{j}:j\geq 2,F=R,I\}.
\end{equation*}
As we do not permit both $X$ and $Y$ to be of type $SU(n)$, without loss of
generality $\Omega _{X}$ contains all the roots $FEe_{1}+e_{j}$ for $2\leq
j\leq n-1$, as well as both $FEe_{1}\pm e_{n},$ and $\Omega _{Y}$ contains
either all $FEe_{1}+e_{j}$ for $2\leq j$ or all $FEe_{1}+e_{j}$ for $2\leq
j\leq n-1$ and $FEe_{1}-e_{n}$. Let $\Omega _{0}$ be the choice of $%
\{FEe_{1}+e_{n}\}$ or $\{FEe_{1}-e_{n}\}$, depending on which belongs to $%
\Omega _{Y}$. Let $g\in G_{n-1}$ be the Weyl conjugate that changes the
signs $2,\dots,n-1$ (and $n$ if needed to be an even sign change). Then $%
Ad(g)(\Omega _{X})\supseteq \{FEe_{1}-e_{j},FEe_{1}\pm e_{n}:j\geq 2\}$ and
hence $\{Ad(g)(\Omega _{X}),\Omega _{Y}\diagdown \Omega _{0}\}=\Omega $.
Take $M=REe_{1}\pm e_{n}$ with the choice of $\pm $ depending on which
belongs to $\Omega _{Y}$.

Case 3: $X$\ and $Y$ are of different dominant type.

Without loss of generality assume $X$ is dominant $SU(m)$ type and $Y$ is
dominant $B_{J},C_{J}$ or $D_{J}$ type, depending on the type of the Lie
algebra. Eligibility implies that $2J+m\leq 2n.$

Let 
\begin{equation*}
\Omega =\{FEe_{1}\pm e_{j},\ (FE(2)e_{1}):j\geq 2\}.
\end{equation*}
(with the inclusion of $FEe_{1}$ if the Lie algebra is type $B_{n}$ or $%
FE(2e_{1})$ if the Lie algebra is $C_{n}$). We have 
\begin{equation*}
\Omega _{X}=\{FEe_{1}+e_{j},FEe_{1}-e_{n}:j<n,\}\text{ if }X=(a,\dots,a,-a) 
\text{ in }D_{n}
\end{equation*}
and 
\begin{equation*}
\Omega _{X}=\{FEe_{1}+e_{j},FEe_{1}-e_{k},(FE(2)e_{1}):j\geq 2,k>m,F=R,I\} 
\text{ otherwise.}
\end{equation*}
Put 
\begin{equation*}
\Omega _{0}=\{FEe_{1}+e_{n-J+1}\}\subseteq \Omega _{X}\cap \Omega _{Y}
\end{equation*}
(or $\Omega _{0}=\{FEe_{1}-e_{n}\}$ if $J=1$ and $X=(a,\dots,a,-a)$ in $%
D_{n})$. Applying a Weyl conjugate from $G_{n-1}$, we can assume 
\begin{equation*}
\Omega _{Y}=\{FEe_{1}\pm e_{j}:2\leq j\leq n-J+1,F=R,I\}.
\end{equation*}

If $n-J+1\geq m$, then we already have 
\begin{equation*}
\{\Omega _{Y},\Omega _{X}\diagdown \Omega _{0}\}=\Omega,
\end{equation*}
so property (i) of the general strategy holds with $g=Id.$ Take $%
M=REe_{1}+e_{n-J+1}$ (resp., take $M=REe_{1}-e_{n})$ to complete the
argument.

Otherwise $m+J-n\geq 2$ (which implies $J\geq 2$). Put 
\begin{equation*}
\Omega _{1}=\{FEe_{1}+e_{k}:2\leq k\leq n-J,F=R,I\}\subseteq (\Omega
_{Y}\cap \Omega _{X})\diagdown \Omega _{0}
\end{equation*}
and define 
\begin{equation*}
H=\sum_{j=2}^{J-1}REe_{j}+e_{n-J+j}+REe_{J-}e_{n}\text{ if }X=(a,\dots,a,-a) 
\text{ in type }D_{n}
\end{equation*}
and 
\begin{equation*}
H=\sum_{j=2}^{m+J-n}REe_{j}+e_{n-J+j}\text{ otherwise.}
\end{equation*}

As $J\neq n$, $e_{j}+e_{n-J+j}$ are roots of the Lie algebra $\mathfrak{g}
_{n-1}$. Let $2\leq k\leq m+J-n$. Observe that $k\neq n-J+j$ for any $j\geq
2 $, for if so, then $j=k-n+J$ $\leq m+2J-2n$ and therefore the eligibility
condition would imply $j\leq 0$. Thus, if $2\leq k\leq m+J-n$, then $%
ad(H)(FEe_{1}+e_{k})=c_{k}FEe_{1}-e_{n-J+k}$ (or $%
ad(H)(FEe_{1}+e_{J})=c_{J}FEe_{1}+e_{n}$ if $X=(a,\dots,-a)$).

The eligibility condition also implies 
\begin{equation*}
\Omega _{1}\supseteq \{FEe_{1}+e_{k}:2\leq k\leq m+J-n\},
\end{equation*}
therefore 
\begin{equation*}
sp\{ad(H)(\Omega _{1})\}\supseteq sp\{FEe_{1}-e_{j},FEe_{1}+e_{n}:n-J+2\leq
j\leq n-1,F=R,I\}
\end{equation*}
if $X=(a,\dots,a,-a)$ in type $D_{n}$ and 
\begin{equation*}
sp\{ad(H)(\Omega _{1})\}\supseteq sp\{FEe_{1}-e_{j}:n-J+2\leq j\leq
m,F=R,I\} \text{ otherwise.}
\end{equation*}

Since $\Omega _{Y}\diagdown \Omega
_{1}=\{FEe_{1}-e_{j},FEe_{1}+e_{n-J+1}:2\leq j\leq n-J+1\}$, in either case
we have 
\begin{equation*}
sp\{ad(H)(\Omega _{1}),\Omega _{Y}\diagdown \Omega _{1},\Omega _{X}\diagdown
\Omega _{0}\}=sp\Omega .
\end{equation*}
By Lemma \ref{genstrategyb} there is some $g\in G_{n-1}$ (namely, $g=\exp tH$
for sufficiently small $t$) such that 
\begin{equation*}
sp\{Ad(g)(\Omega _{Y}),\Omega _{X}\diagdown \Omega _{0}\}=sp\Omega .
\end{equation*}
Again, take $M=REe_{1}+e_{n-J+1}$ and apply the general strategy to complete
the argument.

\textbf{Part II: $g_{n}$ is type $SU(n)$.}

This is very similar to case 1(a). Let 
\begin{equation*}
\Omega =\{FEe_{1}-e_{j}:2<j\leq n,F=R,I\}.
\end{equation*}
We have 
\begin{align*}
\Omega _{X} &=\{FEe_{1}-e_{j}:S_{X}<j\leq n,\text{ }F=R,I\}\text{ and} \\
\Omega _{Y} &=\{FEe_{1}-e_{j}:S_{Y}<j\leq n,\text{ }F=R,I\}.
\end{align*}
Put $\Omega _{0}=\{FEe_{1}-e_{n}:F=R,I\}$. Take $g\in SU(n-1)$ to be the
Weyl conjugate that interchanges the letters $S_{Y}+j$ and $1+j$ for $%
j=1,\dots,S_{X}-1$. The eligibility condition ensures this is well defined
and leaves $1$ and $n$ unchanged. Clearly $\{Ad(g)(\Omega _{Y}),\Omega
_{X}\diagdown \Omega _{0}\}=\Omega $. Take $M=REe_{1}-e_{n}$ and apply the
general strategy in the usual manner.
\end{proof}

\section{Proof of the Main Theorem}

In this section we will complete the proof of Theorem \ref{main}.

\begin{proof}[Proof of Theorem \protect\ref{main}]
\textbf{Necessary conditions for Absolute continuity: }Lemma \ref{Elig}
shows that absolutely continuous tuples are eligible, while in Lemma \ref%
{Excep} we saw that the exceptional tuples, other than possibly the pairs of
type $(SU(n),SU(n-1))$ in the Lie algebra of type $D_{n}$ with $n\geq 6$,
are not absolutely continuous.

The rest of the proof is devoted to establishing that the eligible,
non-exceptional tuples are absolutely continuous.

\textbf{Sufficient conditions for Absolute continuity for Lie types $A_{n}$, 
$B_{n}$ and $C_{n}$:}

Case $L=2$. The proof proceeds by induction on the rank $n$ of the Lie
algebra. We begin $A_{n}$ with $n=1$ (type $SU(2)$) and $B_{n}$ with $n=2$.
Although it is customary to only define $C_{n}$ for $n\geq 3$, there is no
harm in beginning with $C_{2}$, meaning the root system $\pm
\{2e_{1},2e_{2},e_{1}\pm e_{2}\}$, which is Lie isomorphic to $B_{2}$.

According to \cite[Thm.~8.2]{GHMathZ}, all non-zero pairs $(X,Y)$ in the Lie
algebras of type $SU(2)$ and $B_{2}$ have the property that both $\mu
_{X}^{2},\mu _{Y}^{2}\in L^{2}$. Thus $\mu _{X}\ast \mu _{Y}$ is a compactly
supported measure in $L^{2}$ and hence is an absolutely continuous measure.
The existence of $g_{1},g_{2}\in G_{n}$ with 
\begin{equation*}
\sum_{i=1}^{2}T_{Ad(g_{i})(X_{i})}(O_{X_{i}})=\mathfrak{g}_{n}
\end{equation*}
is a Lie isomorphism invariant, thus from Prop.~\ref{key} we can also deduce
that $\mu _{X}\ast \mu _{Y}$ is an absolutely continuous measure for all
non-zero $(X_{1},X_{2})$ in the Lie algebra of type $C_{2}$.

Now, inductively assume that all eligible, non-exceptional pairs in $SU(n-1),
$ $B_{n-1}$ or $C_{n-1}$, with $n\geq 3,$ are absolutely continuous. (Of
course, there are no exceptional pairs in $B_{n-1}$ or $C_{n-1}$.)

Let $(X,Y)$ be an eligible, non-exceptional pair in $SU(n)$, $B_{n}$ or $%
C_{n}$, and form the reduced pair $(X^{\prime },Y^{\prime })$. The reduced
pair is eligible by Lemma \ref{eligible}. Notice that only an element of
type $SU(\frac{n+1}{2})\times SU(\frac{n-1}{2})$ in $SU(n)$ will reduce to
an element of type $SU(\frac{n-1}{2})\times SU(\frac{n-1}{2})$ in $SU(n-1)$.
Furthermore, a pair of elements each of type $SU(\frac{n+1}{2})\times SU( 
\frac{n-1}{2})$ is not eligible in $SU(n)$, thus we can assume $(X^{\prime
},Y^{\prime })$ is both eligible and non-exceptional. By the induction
assumption, $(X^{\prime },Y^{\prime })$ is an absolutely continuous pair.
But then the induction step, Prop.~\ref{indstep}, implies that $(X,Y)$ is
absolutely continuous.

Case $L\geq 3$. Again, we proceed by induction on $n$. We remark that as $%
\mu \ast \nu $ is absolutely continuous if $\mu $ is absolutely continuous
and $\nu $ is an arbitrary measure, the fact that the convolution of any two
non-zero orbital measures in type $SU(2),B_{2}$ or $C_{2}$ is absolutely
continuous, proves that the same is true for the convolution of any $L$
non-zero orbital measures. This starts the induction.

First, suppose $(X_{1},\dots,X_{L})$ is an eligible $L$-tuple in $B_{n}$ or $%
C_{n}$ with $n\geq 3$. We will let $\Omega $ be as in Prop.~\ref{indstep},
depending on whether $\mathfrak{g}$ is type $B_{n}$ or $C_{n}$, 
\begin{equation*}
\Omega =\{FEe_{1}\pm e_{j},FE(2)e_{1}:j=2,\dots,n,F=R,I\}.
\end{equation*}
As a pair of elements that is dominant $SU$ type in $B_{n}$ or $C_{n}$ is
eligible and not exceptional, the theorem for $L=2$ implies the convolution
of (even) their two orbital measures is absolutely continuous. Thus we may
assume that at most one $X_{i}$ is dominant $SU$ type.

Suppose that no $X_{i}$ are dominant $SU$ type and form the corresponding $%
X_{i}^{\prime }$. If $X_{i}^{\prime }$ and $X_{j}^{\prime }$ are dominant $SU
$, then the pair $(X_{i},X_{j})$ is eligible (and not exceptional), thus $%
\mu _{X_{i}}\ast \mu _{X_{j}}$ is absolutely continuous. Hence we can assume
that at most one $X_{i}^{\prime }$ is dominant $SU$ type.

Since $S_{X^{\prime }}=S_{X}-2$ when both $X$ and $X^{\prime }$ are dominant 
$B$ (or $C)$ type it follows that 
\begin{equation*}
\sum_{i=1}^{L}S_{X_{i}^{\prime }}\leq \sum_{i=1}^{L}S_{X_{i}}-2(L-1)\leq
2n(L-1)-2(L-1)=2(n-1)(L-1).
\end{equation*}
This shows that $(X_{1}^{\prime },\dots.,X_{L}^{\prime })$ is eligible in $%
\mathfrak{g}_{n-1}$. As it is not exceptional, the induction assumption
implies it is an absolutely continuous tuple.

Here $\Omega _{X_{i}}=\{FEe_{1}\pm e_{j}:j>J_{i}\}$ where $2J_{i}=S_{X_{i}}$%
. Taking $g_{i}$ to be the Weyl conjugate that switches appropriate letters
(and fixes the letters $1$ and $n$) we can arrange for 
\begin{equation*}
Ad(g_{i})(\Omega _{X_{i}})=\{FEe_{1}\pm
e_{j}:j=(i-1)n-\sum_{k=1}^{i-1}J_{k}+2,\dots,in-\sum_{k=1}^{i}J_{k}+1\}
\end{equation*}
(with suitable modifications if any of the specified choices of $j$ exceed $n
$).

If $(L-1)n-\sum_{k=1}^{L-1}J_{k}+1\geq n$, then 
\begin{equation*}
\bigcup\limits_{i=1}^{L-1}Ad(g_{i})\left( \Omega _{X_{i}}\right)
=\{FEe_{1}\pm e_{j}:j=2,\dots,n\},
\end{equation*}
and this coincides with the set $\Omega _{Y}$ for a suitable $Y$ of type $%
B_{1}$ (or $C_{1})$ (meaning type $B_{1}$ (or $C_{1}$) $\times SU(1)\times
\cdots \times SU(1))$. As always $S_{X}\leq 2(n-1)$, the pair $(Y,X_{L})$ is
eligible.

Otherwise, if we let $m=n-(L-1)n+\sum_{i=1}^{L-1}J_{i}$ and take a suitable
choice of $Y$ of type $B_{m}$ (or $C_{m})$, then 
\begin{equation*}
\bigcup\limits_{i=1}^{L-1}Ad(g_{i})\Omega _{X_{i}}=\Omega _{Y}.
\end{equation*}
The eligibility condition ensures that 
\begin{align*}
S_{Y}+S_{X_{L}} &=2\left( n-(L-1)n+\sum_{i=1}^{L-1}J_{i}\right) +2J_{L} \\
&\leq 2n-2(L-1)n+\sum_{i=1}^{L}S_{X_{i}}\leq 2n
\end{align*}
and thus the pair $(Y,X_{L})$ is eligible and clearly not exceptional. The
arguments given in the proof of Prop.~\ref{indstep} Case 1 show that then
there is some $g\in G_{n-1}$, $M\in \mathfrak{g}_{n}$ and $\Omega
_{0}\subseteq \Omega _{X_{L}}$ such that

(i) $sp\Omega =sp\{Ad(g)(\Omega _{Y}),\Omega _{X_{L}}\backslash \Omega
_{0}\};$

(ii) $ad^{k}(M):\mathcal{N}_{X_{L}}\backslash \Omega _{0}\rightarrow
sp\{\Omega ,\mathfrak{g}_{n-1}\}$ for all positive integers $k$; and

(iii) The span of the projection of $Ad(\exp sM)(\Omega _{0})$ onto the
orthogonal complement of $sp\{\mathfrak{g}_{n-1},\Omega \}$ in $\mathfrak{g}
_{n}$ is a surjection for all small $s>0.$

Since 
\begin{equation*}
sp\{Ad(g)(\Omega _{Y}),\Omega _{X_{L}}\backslash \Omega _{0}\}\subseteq
sp\{Ad(gg_{i})(\Omega _{X_{i}}),\Omega _{X_{L}}\backslash \Omega
_{0}:i=1,\dots.,L-1\}
\end{equation*}
we can call upon the general strategy, Prop.~\ref{general strategy}, with $%
g_{i}$ replaced there by $gg_{i},$ to deduce that $(X_{1},\dots,X_{L})$ is
an absolutely continuous tuple. This completes the argument when no $X_{i}$
are of dominant $SU$ type.

Otherwise, there is one $X_{i}$ which is of dominant $SU$ type, say $X_{L}$.
If there is another index $j$ such that $X_{j}^{\prime }$ is of dominant $SU$
type, then the pair $(X_{L},X_{j})$ is eligible and not exceptional, hence $%
\mu _{X_{L}}\ast \mu _{X_{j}}$ is absolutely continuous.

So we may assume all $X_{j}^{\prime }$ with $j\neq L$ are of dominant $B$
(or $C)$ type. Thus 
\begin{equation*}
\sum S_{X_{i}}^{\prime }\leq \sum S_{X_{i}}-2(L-1)\leq 2(n-1)(L-1),
\end{equation*}
so $(X_{1}^{\prime },\dots,X_{L}^{\prime })$ is an eligible $L$-tuple.
Again, taking $g_{i}$ to be suitable Weyl conjugates we have 
\begin{equation*}
\bigcup\limits_{i=1}^{L-1}Ad(g_{i})(\Omega _{X_{i}})=\{FEe_{1}\pm
e_{j}:j=2,\dots,(L-1)n-\sum_{i=1}^{L-1}J_{i}+1\}
\end{equation*}
and if we let $Y$ be of type $B_{m}$ where $m=n-(L-1)n+\sum_{i=1}^{L-1}J_{i},
$  then 
\begin{equation*}
\bigcup\limits_{i=1}^{L-1}Ad(g_{i})(\Omega _{X_{i}})=\Omega _{Y}.
\end{equation*}
The eligibility condition ensures that 
\begin{equation*}
S_{Y}+S_{X_{L}}=2\left( n-(L-1)n+\sum_{i=1}^{L-1}J_{i}\right) +S_{X_{L}}\leq
2n,
\end{equation*}
so the pair $(Y,X_{L})$ is eligible. Complete the proof using the arguments
of Prop.~\ref{indstep}, but this time using Case 3 as $X_{L}$ and $Y$ are of
opposite dominant types.

The argument is similar, but easier, if the Lie algebra is type $SU(n)$. We
first check that $(X_{1}^{\prime },\dots,X_{L}^{\prime })$ is eligible when $%
(X_{1},\dots,X_{L})$ is eligible. This is clear if at most one $X_{i}$ has $%
S_{X_{i}}=S_{X_{i}^{\prime }}$. If two or more $X_{i}$ have $%
S_{X_{i}}=S_{X_{i}^{\prime }}$, then these two satisfy $S_{X_{i}}\leq n/2$
and because all $S_{X^{\prime }}\leq n-2$, we have 
\begin{equation*}
\sum_{i=1}^{L}S_{X_{i}^{\prime }}\leq 2(n/2)+(L-2)(n-2)\leq (L-1)(n-1),
\end{equation*}
proving eligibility.

Set $\Omega =\{FEe_{1}-e_{j}:2\leq j\leq n\}$. We have $\Omega
_{X_{i}}=\{FEe_{1}-e_{j}:j>S_{X_{i}}\}$. Upon taking $g_{i}$ suitable Weyl
conjugates that permute letters we obtain 
\begin{equation*}
\bigcup\limits_{i=1}^{L-1}Ad(g_{i})\Omega _{X_{i}}=\Omega _{Y}.
\end{equation*}
where $Y$ is an element of the torus of $SU(n)$ of type $SU(m)$ with $%
m=n-(L-1)n+\sum_{i=1}^{L-1}J_{i}$. The eligibility assumption ensures $%
(X_{L},Y)$ is an eligible pair and it is clearly not exceptional. Now
complete the argument using the $L=2$ case in the same manner as for type $%
B_{n}$ and $C_{n}$.
\end{proof}

The many exceptional pairs in $D_{n},$ ($n=4$ in particular), cause
complications in proving the theorem for type $D_{n}$. We will again prove
the main theorem by an induction argument for $L=2$, but it will be
convenient to begin the argument with type $D_{5}$. In the next lemma we
will prove that all eligible, non-exceptional pairs in $D_{4}$ and $D_{5}$
are absolutely continuous. This will start the base case for us.

We will actually begin with $D_{3}$. Usually $D_{n}$ is defined for $n\geq 4$%
, but that is because $D_{3}$ is Lie isomorphic to type $A_{3}$. As the
problem of characterizing the $L$-tuples in type $A_{3}$ has already been
done we can use this characterization, together with the induction step,
Prop.~\ref{indstep}, to handle most of the eligible, non-exceptional pairs
in $D_{4}$ and $D_{5}$. This approach will work whenever the reduced pair is
known to be an absolutely continuous pair (in $D_{3}$ or $D_{4}$,
respectively). There will still be a few remaining pairs to consider and
these will be handled directly by verifying Wright's criteria for absolute
continuity, Thm.~\ref{WrCriteria}.

\begin{lemma}
\label{DnBaseCase}All the eligible, non-exceptional pairs in $D_{4}$ and $%
D_{5}$ are absolutely continuous.
\end{lemma}

\begin{proof}
As explained above, we begin the proof by considering $D_{3}$. Under the Lie
isomorphism between $D_{3}$ and $A_{3}$, any subsystem of type $D_{2}$ in $%
D_{3}$ is isomorphic to one of type $A_{1}\times A_{1}$, type $D_{1}$ is
isomorphic to one of type $A_{0}$ (or $SU(1)$) and types $SU(j)$ for $%
j=1,2,3 $ are unchanged under such an isomorphism. With this observation and
the criteria for absolute continuity already known for the Lie algebra of
type $A_{3}$, it is easy to check that all pairs $(X,Y)$ in $D_{3}$ are
absolutely continuous except those of type $(D_{2},D_{2})$, $(D_{2},SU(3))$, 
$(SU(3),SU(3))$ and $(SU(3),SU(2))$, the first two of these being not
eligible and the latter two, exceptional.

Case $D_{4}$: Prop.~\ref{indstep} guarantees that all eligible,
non-exceptional pairs, $(X,Y)$, in $D_{4}$ are absolutely continuous, except
when the reduced pair, $(X^{\prime },Y^{\prime }),$ is one of the four pairs
listed above. Furthermore, because we have already seen that the pair $%
(X^{\prime },Y^{\prime })$ is eligible whenever $(X,Y)$ is an eligible,
non-exceptional pair, we will only need to give a special argument for those
pairs $(X,Y)$ where $X^{\prime }$ is type $SU(3)$ and $Y^{\prime }$ is
either type $SU(3)$ or $SU(2)$ (the latter being type $SU(2)\times D_{1}$ or 
$SU(2)\times SU(1)$).

Thus we are left to study the pairs $(X,Y)$ where $X$ is of type $SU(4)$ and 
$Y$ is one of type $SU(4)$, type $SU(3)$ (to be more precise, either type $%
SU(3)\times D_{1}$ or $SU(3)\times SU(1)$), type $SU(2)\times D_{2}$ or $%
SU(2)\times SU(2)$. However, these are all exceptional pairs except when $X$
is type $SU(4)$, $Y$ is of type $SU(2)\times SU(2)$ and $\Phi _{Y}$ is not
Weyl conjugate to a subset of $\Phi _{X}$.

To prove this last pair is absolutely continuous, we verify the criteria of
Thm.~\ref{WrCriteria} (with $X_{1}=X$ and $X_{2}=Y$) and follow the notation
there. Here we have $|\Phi |=24$ and $|\Phi _{X_{1}}|+|\Phi _{X_{2}}|=
12+4=16$. The rank $3$, root subsystems, $\Psi $ of $D_{4}$, are those of
type $D_{3}$, $SU(4)$ (two non-Weyl conjugate subsystems) and $D_{2}\times
SU(2)$.

When $\Psi $ is type $D_{2}\times SU(2)$, then $|\Psi |=6$. Thus we even
have $|\Phi |-|\Psi |-1\geq |\Phi _{X_{1}}|+|\Phi _{X_{2}}|,$ so (\ref{WrC})
clearly holds. When $\Psi $ is type $D_{3}$, then $|\Psi |=12$. However, $%
\left\vert \Phi _{X_{1}}\cap \sigma (\Psi )\right\vert =6$ and $\left\vert
\Phi _{X_{2}}\cap \sigma (\Psi )\right\vert \geq 2$ for all choices of $%
\sigma \in W$ because $\sigma (\Psi )$ must contain $\pm e_{i}\pm e_{j},\pm
e_{i}\pm e_{k},\pm e_{j}\pm e_{k}$ for three choices of letters $i,j,k$.
Thus the LHS of (\ref{WrC}) is $11$, while the RHS is at most $8$.

Now assume $\Psi $ is type $SU(4)$. First, suppose $\Psi $ is Weyl conjugate
to the set of annihilators of $X$. Since we need to calculate the
intersection of $\Phi _{X_{j}}$ with all Weyl conjugates of $\Psi $, there
is no loss of generality in assuming $\Phi _{X_{1}}=\Psi
=\{e_{i}-e_{j}:1\leq i\neq j\leq 4\}$. By assumption, $\Phi _{X_{2}}$ is not
Weyl conjugate to a subset of $\Phi _{X_{1}}$, thus there is also no loss of
generality in assuming $\Phi _{X_{2}}=\{\pm (e_{1}-e_{2}),\pm
(e_{3}+e_{4})\} $.

The reader can check that $\left\vert \Phi _{X_{1}}\cap \sigma (\Psi
)\right\vert $ is minimal when we take the choice of $\sigma \in W$ that
switches two signs and in this case $\left\vert \Phi _{X_{1}}\cap \sigma
(\Psi )\right\vert =4$. Similarly, it can be shown that if $\sigma $ is any
Weyl conjugate, then $\left\vert \Phi _{X_{2}}\cap \sigma (\Psi )\right\vert
\geq 2$, so that again the LHS of (\ref{WrC}) is $11$ and the RHS is at most 
$10$.

Finally, suppose $\Psi $ is not Weyl conjugate to $\Phi _{X_{1}}$. Without
loss of generality we can assume $\Psi $ is as before and $\Phi
_{X_{1}}=\{e_{i}-e_{j},\pm (e_{4}+e_{j}):1\leq i\neq j\leq 3\}$. Again $%
\left\vert \Phi _{X_{1}}\cap \sigma (\Psi )\right\vert $ is minimal when $%
\sigma $ is the Weyl element that switches two signs, but in this case $%
\left\vert \Phi _{X_{1}}\cap \sigma (\Psi )\right\vert =6$. This is already
enough to establish (\ref{WrC}) and completes the argument that $(X,Y)$ is
an absolutely continuous pair.

Case $D_{5}$: Again, Prop.~\ref{indstep} implies we only need to study the
eligible non-exceptional pairs, $(X,Y)$ in $D_{5},$ where the reduced pair
has $X^{\prime }$ of type $SU(4)$ and $Y^{\prime }$ one of type $SU(4)$, $%
SU(3)$, $SU(2)\times D_{2}$, or $SU(2)\times SU(2)$. Since the pairs $%
(SU(5),SU(5)$ and $(SU(5),$ $SU(4))$ are exceptional and the pair $%
(SU(5),D_{3}\times SU(2))$ is not eligible, this reduces the problem to the
study of the pairs $(X,Y)$ where $X$ is of type $SU(5)$ and $Y$ is either of
type $SU(3)\times D_{2}$ or $SU(3)\times SU(2)$. Further, since the set of
annihilators of an element of type $SU(3)\times SU(2)$ is contained in the
set of annihilators of an element of type $SU(3)\times D_{2}$, it will
suffice to prove that the pair $(SU(5)$, $SU(3)\times D_{2})$ is absolutely
continuous.

For this, we again use Thm.~\ref{WrCriteria}. In $D_{5}$, the rank $4$ root
subsystems $\Psi $ which we must study are those of type $D_{4}$, $SU(5)$, $%
D_{3}\times SU(2)$ and $D_{2}\times SU(3)$, with cardinalities $24$, $20$, $%
14$ and $10$, respectively. The cardinality of $\Phi $ is $40$, while $|\Phi
_{X_{1}}|=20$ and $|\Phi _{X_{2}}|=10$.

Let $\Lambda $ be a root subsystem of type $D_{2},D_{3}$ or $D_{4}$ in $%
D_{5} $. It is easy to see that if $\Lambda $ is a root subsystem of type $%
D_{j}$ in $D_{5},$ with $j=2,3,4$, then $\left\vert \Phi _{X_{1}}\cap
\Lambda \right\vert =\frac{1}{2}\left\vert \Lambda \right\vert $.
Furthermore, $\left\vert \Phi _{X_{2}}\cap \Lambda \right\vert =6$ whenever $%
\Lambda $ is type $D_{4}$. Since the action of a Weyl element preserves the
type of a root subsystem, these calculations can be used to show that (\ref%
{WrC}) holds if $\Psi $ is type $D_{4}$, $D_{3}\times SU(2)$ or $D_{2}\times
SU(3)$.

When $\Psi $ is type $SU(5)$, then $\left\vert \Phi _{X_{1}}\cap \sigma
(\Psi )\right\vert $ $\geq 8$ for all $\sigma $ (with the minimum occurring
when $\sigma $ is two sign changes). Moreover, $\left\vert \Phi _{X_{2}}\cap
\sigma (\Psi )\right\vert \geq 4$ so that again (\ref{WrC}) is satisfied.
This shows that the pair $(SU(5)$, $SU(3)\times D_{2})$ is absolutely
continuous and completes the base case arguments.
\end{proof}

Further complications arise with type $D_{n}$ because of the fact that when $%
X$ is of type $SU(n)$, then $\mu _{X}^{2}$ $\ $is not absolutely continuous.
We have already seen this complication in the proof of Prop.~\ref{indstep}
(when $L=2$), but it presents further difficulties when $L>2$. To handle
this, we introduce the following terminology for the remainder of the proof.

\begin{definition}
We will say that $X$ is \textbf{almost dominant $SU$ type} if $X$ is type $%
D_{J}\times SU(s_{1})\times \cdots \times SU(s_{t})$ where $J\leq \sum s_{i}$%
.
\end{definition}

Of course, if $X$ is dominant $SU$ type, then it is almost dominant $SU$
type. However, $X$ is also almost dominant $SU$ type if $X$ is dominant $D$
type, but $X^{\prime }$ is dominant $SU$ type, for instance. If $X$ is
almost dominant $SU$ type and not type $SU(n)$, then \cite[Thm.~8.2]{GHMathZ}%
, implies $\mu _{X}^{2}\in L^{2}$. Here are some additional properties.

\begin{lemma}
\label{D}Suppose $X_{1},X_{2}$ are almost dominant $SU$ type in $D_{n}$ and $%
X_{3}\neq 0$.

\textrm{(i)} If neither $X_{1}$ nor $X_{2}$ are type $SU(n)$, then $\mu
_{X_{1}}\ast \mu _{X_{2}}\in L^{2}.$

\textrm{(ii)} If $X_{1}$ and $X_{2}$ are both type $SU(n)$ and $X_{3}$ is
not, then $\mu _{X_{1}}\ast \mu _{X_{2}}\ast \mu _{X_{3}}\in L^{2}.$

\textrm{(iii)} More generally, if $X_{1}$ is type $SU(n)$ and $X_{2}$ is
not, then $\mu _{X_{1}}\ast \mu _{X_{2}}\ast \mu _{X_{3}}\in L^{2}.$

\textrm{(iv)} If $n\geq 5$ \textrm{(}or $n=4$\textrm{)} and $X_{3}$ \textrm{(%
}and $X_{4}$\textrm{)} is almost dominant $SU(n)$ type, then $\mu
_{X_{1}}\ast \mu _{X_{2}}\ast \mu _{X_{3}}(\ast \mu _{X_{4}})\in L^{2}.$
\end{lemma}

\begin{proof}
(i) follows from \cite{GHMathZ} as remarked above. The fact that it is
absolutely continuous, which is actually all we will need for our
application, also follows from the $L=2$ part of the proof of the main
theorem since $(X_{1},X_{2})$ is an eligible, non-exceptional pair.

(iv) holds similarly from \cite{GHMathZ} since $\mu _{X}^{3}\in L^{2}$
whenever $X$ is almost dominate $SU$ type and $n\geq 5,$ and $\mu
_{X}^{4}\in L^{2}$ when $n=4$. (Alternatively, absolute continuity can be
checked from Theorem \ref{WrCriteria}.)

For (ii) and (iii) we proceed by induction on $n$, noting that according to
the main theorem, as already established for all $L\geq 2$ in the Lie
algebras of type $A_{n}$, all triples in $A_{3}$ (equivalently, $D_{3}$) are
absolutely continuous, except when all three are type $SU(3)$.

Now assume $n\geq 4$. We put 
\begin{equation*}
\Omega =\{FEe_{1}\pm e_{j}:j=2,\dots,n,F=R,I\}.
\end{equation*}

(ii): Here $X_{1}^{\prime },X_{2}^{\prime }$ will be of type $SU(n-1)$ in $%
D_{n-1}$, while $X_{3}^{\prime }$ is not, so the induction hypothesis
applies. Without loss of generality we can assume 
\begin{equation*}
\Omega _{X_{1}}=\{FEe_{1}+e_{j}:j\geq 2,F=R,I\}
\end{equation*}
and $\Omega _{X_{2}}$ either coincides with $\Omega _{X_{1}}$ or 
\begin{equation*}
\Omega _{X_{2}}=\{FEe_{1}+e_{j},FEe_{1}-e_{n}:j\leq n-1,F=R,I\}.
\end{equation*}

As $X_{3}$ is not of type $SU(n)$, $\Omega _{X_{3}}$ contains $\{FEe_{1}\pm
e_{n}\}$. Let $g$ be the Weyl conjugate changing the signs of $2,\dots,n-1$
(and $n$ if needed to be an even sign change). Then $\Omega _{X_{1}}\cup
Ad(g)(\Omega _{X_{2}})$ contains all of $\Omega $ except for possibly $%
\{FEe_{1}-e_{n}:F=R,I\}$. If $\Omega _{0}=$ $\{FEe_{1}+e_{n}\}$, we have 
\begin{equation*}
\Omega _{X_{1}}\cup Ad(g)(\Omega _{X_{2}})\cup (\Omega _{X_{3}}\diagdown
\Omega _{0})=\Omega .
\end{equation*}
Taking $M=REe_{1}+e_{n}$ one can verify that the hypotheses of the general
strategy, Prop.~\ref{general strategy}, are all satisfied. Consequently, $%
(X_{1},X_{2},X_{3})$ is absolutely continuous.

(iii): We define $X_{1}^{\prime }$ and $X_{3}^{\prime }$ as usual, but will
re-define $X_{2}^{\prime }$ so that it continues to be of almost dominant $SU
$ type and not type $SU(n-1)$ (so that we will be able to apply the
induction hypothesis). This can be achieved by defining $X_{2}^{\prime }$ to
be type $D_{J-1}\times SU(s_{1})\times \cdots \times SU(s_{t})$ if $X_{2}$
is type $D_{J}\times SU(s_{1})\times \cdots \times SU(s_{t})$ with $J>1$, or
defining $X_{2}^{\prime }$ to be type $D_{J}\times SU(s_{1-1})\times \cdots
\times SU(s_{t})$ if $J=0$ or $1$ and $s_{1}=\max s_{j}$. The fact that $%
X_{2}$ is almost dominant $SU$ type ensures that $J\leq n/2$, so whether $%
X_{2}$ is dominant $SU$ type or not, $S_{X_{2}}\leq n$ and thus $%
(X_{1},X_{2})$ is an eligible pair. Further, $X_{1},X_{2}$ are not both of
type $SU(n)$.

The arguments given in Prop.~\ref{indstep}, Case 2 or 3 depending on the
situation, can be applied to prove there is some $g\in D_{n-1}$ such that $%
sp\{\Omega _{X_{1}}\diagdown \Omega _{0},Ad(g)(\Omega _{X_{2}})\}=sp\Omega $
where $\Omega _{0}$ is taken to be the choice of $FEe_{1}+e_{n}$ or $%
FEe_{1}-e_{n}$ that belongs to $\Omega _{X_{1}}$. Therefore 
\begin{equation*}
sp\{\Omega _{X_{1}}\diagdown \Omega _{0},Ad(g)\left( \Omega _{X_{2}}\right),
\Omega _{X_{3}}\}=sp\Omega .
\end{equation*}
Now take $M=REe_{1}\pm e_{n}$ (depending on the choice of $\Omega _{0}$) and
apply the general strategy.
\end{proof}

We are now ready to conclude the proof of Theorem \ref{main} by completing
the proof of sufficiency for absolute continuity in type $D_{n}$.

\begin{proof}[Proof of Theorem \protect\ref{main} continued]
\textbf{Sufficient conditions for Absolute continuity for Lie type $D_{n}$:}

Case $L=2$. Lemma \ref{DnBaseCase} starts the induction argument for type $%
D_{n}$. Now, inductively assume that all eligible, non-exceptional pairs in $%
D_{n-1}$, with $n\geq 6$, are absolutely continuous. By Lemma \ref{eligible}%
, the pair $(X^{\prime },Y^{\prime })$ is eligible. If it is an exceptional
pair, then it must be either of type $(SU(n-1)$, $SU(n-1))$ or type $%
(SU(n-1) $, $SU(n-2))$ (where the $SU(n-2)$ could be type $SU(n-2)\times
D_{1}$ or $SU(n-2)\times SU(1))$. But then $(X,Y)$ must also have been an
exceptional pair in $D_{n}$, which is a contradiction. By the induction
assumption, $(X^{\prime },Y^{\prime })$ is an absolutely continuous pair and
hence Prop.~\ref{indstep} implies that $(X,Y)$ is absolutely continuous.

Case $L\geq 3$. Again, we give an induction argument. The base case, $D_{4}$%
, will be discussed at the conclusion of the proof. So assume $n\geq 5$ and $%
(X_{1},\dots,X_{L})$ is an eligible $L$-tuple in $D_{n}$. We note that there
are no exceptional $L$-tuples in $D_{n}$ for $n\geq 5$ when $L\geq 3$.

We will take 
\begin{equation*}
\Omega =\{FEe_{1}\pm e_{j}:j=2,\dots,n\}.
\end{equation*}
More care is needed in this situation then for the Lie algebras of type $%
B_{n}$ and $C_{n},$ since the fact that $\mu _{X}^{2}$ $\notin L^{2}\ $ when 
$X$ is of type $SU(n)$ means, for example, that we cannot immediately assume
that at most one $X_{i}$ is dominant $SU$ type, as we did in the argument
for those Lie types. Here is where Lemma \ref{D} will be useful.

If three or more $X_{i}$ are dominant $SU$ type, then the induction argument
is not even necessary as Lemma \ref{D} (iv) implies that their convolution
is already in $L^{2}$ and hence is absolutely continuous.

If two $X_{i}$ (say, $X_{1},X_{2})$ are both dominant $SU$ type and some $%
X_{j}$, say $X_{3},$ is not, then we call upon one of the first three parts
of the lemma.

So we can assume there is at most one $X_{i}$ that is dominant $SU$ type,
say $X_{1}$. If there is some $X_{j},$ other than $X_{1}$, with $%
X_{j}^{\prime }$ of dominant $SU$ type, then $X_{j}$ is almost dominant $SU$
and not type $SU(n)$. Apply the appropriate part of Lemma \ref{D} with $X_{1}
$, $X_{2}$ equal to this $X_{j},$ and $X_{3}$ any other $X_{i}$.

If all $X_{j}^{\prime },$ other than $j=1,$ remain dominant $D$ type, then
the calculations used in the type $B_{n}$ or $C_{n}$ case show that $%
(X_{1}^{\prime },\dots,X_{L}^{\prime })$ is an eligible, non-exceptional
tuple in $D_{n-1}$. For the induction step we argue in the same fashion as
we did for the Lie algebras of type $B_{n}$ or $C_{n}$ in the same situation.

Finally, assume all $X_{j}$ are dominant $D$ type. If two or more $%
X_{j}^{\prime }$ are dominant $SU$ type, then the corresponding two $X_{j}$
are almost dominant $SU$ type and not of type $SU(n)$. Their convolution is
even in $L^{2}$. If at most one $X_{j}^{\prime }$ is dominant $SU$ type,
then $(X_{1}^{\prime },\dots,X_{L}^{\prime })$ is an eligible tuple, so the
induction hypothesis applies. The induction step is the same as for the
corresponding situation with type $B_{n}$ or $C_{n}$.

To conclude, we must establish the base case, $n=4$. Since $\mu _{X}^{4}\in
L^{2}$ for any non-trivial $X$ in the Lie algebra of type $D_{4}$, every $L$%
-tuple with $L\geq 4$ is an absolutely continuous tuple.

So we can assume $L=3$. The induction argument above can be applied to $%
(X_{1},X_{2},X_{3})$ provided at most one $X_{j}$ is dominant $SU$ type,
hence for such triples it suffices to check that $(X_{1}^{\prime
},X_{2}^{\prime },X_{3}^{\prime })$ in $D_{3}$ is an absolutely continuous
triple. But this follows from the main theorem for type $A_{n}$, since all
triples in $A_{3},$ except when all $X_{i}$ are type $SU(3)$, are absolutely
continuous.

If two or three $X_{j}$ are dominant $SU$ type, but at least one $X_{i}$ is
not of type $SU(4)$, Lemma \ref{D} gives the result.

If all three $X_{i}$ are type $SU(4)$ and their annihilating root systems
are Weyl conjugate, then the triple, $(X_{1},X_{2},X_{3})$, is exceptional.
Thus we can assume the annihilating root systems are not Weyl conjugate. As
the arguments are symmetric, there is no loss of generality in assuming that
the set of annihilating roots for $X_{1}$ coincides with that of $X_{2}$ and
is given by 
\begin{equation*}
\Phi _{X_{1}}=\Phi _{X_{2}}=\{e_{i}-e_{j}:1\leq i\neq j\leq 4\},
\end{equation*}
while 
\begin{equation*}
\Phi _{X_{3}}=\{e_{i}-e_{j},\pm (e_{4}+e_{k}):1\leq i\neq j\leq 3,k=1,2,3\}.
\end{equation*}
We will again call upon Thm.~\ref{WrCriteria} to check the absolute
continuity of the triple. The root systems $\Psi $ of rank $3$ that we must
consider are those of type $D_{3}$, $SU(4)$ (two non-Weyl conjugate root
subsystems) and $D_{2}\times SU(2)$.

Of course, $\sum_{i=1}^{3}\left\vert \Phi _{X_{i}}\right\vert =36$ and $%
\left\vert \Phi \right\vert =24$ when $\Phi $ is the root system of $D_{4}$.
When $\Psi $ is type $D_{2}\times SU(2)$, then $\left\vert \Psi \right\vert
=6$, and as $\Psi $ intersects non-trivially any root subsystem of type $%
SU(4)$, the inequality (\ref{WrC}) is clear. When $\Psi $ is type $D_{3}$ it
is easy to see that $\left\vert \Phi _{X_{i}}\cap \sigma (\Psi )\right\vert
\geq 6$ for each $i$ and any Weyl element $\sigma $.

When $\Psi $ is type $SU(4)$, then $\left\vert \Phi _{X_{i}}\cap \sigma
(\Psi )\right\vert \geq 4$. However, as noted in the $L=2$ argument, if $%
\Psi $ and $\Phi _{X_{i}}$ are non-Weyl conjugate subsystems of type $SU(4)$%
, then this lower bound can be improved to $\dot{6}$. Consequently, $%
2(\left\vert \Phi \right\vert -\left\vert \Psi \right\vert )-1=23$, while
the right hand side of (\ref{WrC}) is at most $36-(4+4+6)=22$, so the
inequality holds.

This completes the base case argument and hence the proof of Theorem \ref%
{main}.
\end{proof}

\section{Applications}

\subsection{Consequences of the Main Theorem}

An element $X\in \mathfrak{t}_{n}$ is said to be \textit{regular} if its set
of annihilating roots is empty. These would be the elements of type $%
SU(1)\times \cdots \times SU(1)$ (in any Lie algebra) or $D_{1}\times
SU(1)\times \cdots \times SU(1)$ in type $D_{n},$ and hence have $S_{X}=1$
or $2$. In \cite{GHGeneric} it was shown that the convolution of the orbital
measures of any two regular elements is absolutely continuous. The methods
used there could be used to prove, more generally, that the convolution of
any orbital measure with the orbital measure of a regular element is
absolutely continuous. Our theorem shows that more is true.

\begin{corollary}
Let $X,Y$ be non-zero elements in the Lie algebra of type $B_{n}$, $C_{n}$,
or $D_{n}$. If $S_{Y}$ $\leq 2,$ then $\mu _{X}\ast \mu _{Y}$ is absolutely
continuous and $O_{X}+O_{Y}$ has non-empty interior, except if $(X,Y)$ is
the exceptional pair $(SU(4),SU(2)\times SU(2))$ in $D_{4}$ where the
annihilating roots of $Y$ are a subset of a Weyl conjugate of those of $X$.
\end{corollary}

\begin{proof}
This is immediate from the theorem since any non-zero $X$ has $S_{X}\leq
2(n-1)$ (with equality only if $X$ is type $B_{n-1}$ ($C_{n-1}$ or $D_{n-1}$
)).
\end{proof}

\begin{corollary}
If $(X_{1},\dots,X_{L})$ is an eligible, non-exceptional $L$-tuple of
matrices in any of the classical Lie algebras, then there are unitarily
similar matrices, $g_{i}^{-1}X_{i}g_{i}$, with the property that $%
\sum_{i=1}^{L}g_{i}^{-1}X_{i}g_{i}$ has distinct eigenvalues.
\end{corollary}

\begin{proof}
This follows from the main theorem because any subset of these matrix groups
with non-empty interior must contain an element with distinct eigenvalues.
Indeed, the elements with distinct eigenvalues are dense.
\end{proof}

On the other hand, if $X_{i}$ are matrices in one of the classical Lie
algebras and there are unitarily similar matrices, $g_{i}^{-1}X_{i}g_{i}\in
O_{X_{i}}$, with the property that $\sum_{i=1}^{L}g_{i}^{-1}X_{i}g_{i}$ has
distinct eigenvalues, then $\sum O_{X_{i}}$ contains an element $Y$ with $%
S_{Y}\leq 2$. (Indeed, $Y$ is either type $SU(1)$ or type $B_{1}$, $C_{1}$
or $D_{1}$.) As noted in the first corollary, $O_{X}+O_{Y}$ has non-empty
interior for any $X\neq 0$ and thus $\mu _{X_{1}}\ast \cdots \mu
_{X_{L}}\ast \mu _{X}$ is absolutely continuous for any $X\neq 0$. It would
be interesting to characterize the $L$-tuples for which $\sum O_{X_{i}}$
contains a matrix with distinct eigenvalues.

It is known that any $n$-fold sum of non-trivial orbits in $B_{n}$, $C_{n}$
or $D_{n}$ has non-empty interior. More can be said.

\begin{corollary}
\label{sharp}Let $n\geq 5$. If $X_{i}$ are non-zero elements in $B_{n}$ 
\textrm{(}$C_{n}$ or $D_{n}$\textrm{)} for $i=1,\dots,n-1$, then $%
O_{X_{1}}+\cdots +O_{X_{n-1}}$ has empty interior if and only if all $X_{i}$
are type $B_{n-1} $ \textrm{(}$C_{n-1}$ or $D_{n-1}$\textrm{).}
\end{corollary}

\begin{proof}
Suppose some $X_{i}$, say $X_{n-1}$, is not type $B_{n-1}$. Then $%
S_{X_{n-1}}\leq 2(n-2)$. As all $S_{X_{i}}\leq 2(n-1)$, 
\begin{equation*}
\sum_{i=1}^{n-1}S_{X_{i}}\leq 2(n-1)(n-2)+2(n-2)\leq 2n(n-2).
\end{equation*}
Thus $(X_{1},\dots,X_{n-1})$ is eligible and non-exceptional and hence the
sum of the orbits has non-empty interior. Since all root subsystems of type $%
B_{n-1}$ are Weyl conjugate, the converse follows from the fact that if $X$
is type $B_{n-1}$, then $\mu _{X}^{n-1}$ is not absolutely continuous \cite%
{GHMathZ}.
\end{proof}

We leave it as an exercise for the reader to determine the choice of $n-1$
tuples that are not absolutely continuous when $n\leq 4$ and in type $A_{n}$.

We note that for type $A_{n}$, $B_{n}$ and $C_{n}$ our proof required the
use of \cite{GHMathZ} only to start the induction process. In the proof
given in \cite{GHMathZ} an induction argument was also used and the base
cases were simply done directly. That approach could have been taken here,
as well. For type $D_{n}$ our proof also used \cite{GHMathZ} to establish
that when $X$ was type $SU(n)$, then $\mu _{X}^{3}$ was absolutely
continuous for $n\geq 5$ and $\mu _{X}^{4}$ was absolutely continuous when $%
n=4$. In fact, the argument that was given there for these special types
actually showed that Theorem \ref{WrCriteria} was satisfied. Thus, our
theorem gives another way to deduce the formulas of \cite{GHMathZ}. For
example, we have

\begin{corollary}
Suppose $\mathfrak{g}$ is type $B_{n}$ and $X$ is type $B_{J}\times
SU(s_{1})\times \cdots \times SU(s_{m})$.

\textrm{(i)} If $X$ is dominant $B$ type, then $\mu _{X}^{L}$ is absolutely
continuous \textrm{(}and $(L)O_{X}$ has non-empty interior\/\textrm{)} if
and only if $L\geq n/(n-J)$.

\textrm{(ii)} If $X$ is dominant $SU$ type, then $\mu _{X}^{2}$ is
absolutely continuous.
\end{corollary}

Similar statements can be made for the other types, taking into account the
exceptional cases.

\begin{remark}
(i) We have not been able to determine if the pair of type $(SU(n)$, $%
SU(n-1))$ in $D_{n}$ for $n\geq 6$ is absolutely continuous. Computer
results suggest that it is not for at least $n=6,7$. We remark that Prop.~%
\ref{indstep} shows that if such a pair is absolutely continuous for any $n$%
, say $n=n_{0}$ then, being an eligible pair, it is absolutely continuous
for all $n>n_{0}$.

(ii) It remains open to solve the analogous problem in the exceptional Lie
algebras, those of type $G_{2}$, $F_{4}$, $E_{6}$, $E_{7}$, or $E_{8}$. In 
\cite{HJY} the minimal $k(X)$ so that $\mu _{X}^{k(X)}$ is absolutely
continuous was determined for each $X$ in the compact exceptional Lie
algebras. Although the abstract root theory machinery can be applied in this
setting, there is no underlying classical matrix algebra from which to
derive the necessary conditions.
\end{remark}

\subsection{Orbital measures on Conjugacy classes in Compact Lie Groups}

A related, but more challenging problem, is to determine which $L$-tuples,
\linebreak $(x_{1},\dots,x_{L})\in G^{L},$ have the property that $\mu
_{x_{1}}\ast \cdots \ast \mu _{x_{L}}$ is absolutely continuous with respect
to Haar measure on the group $G$, when $\mu _{x}$ is the probability
measure, invariant under the conjugation action of $G$ on itself, and
supported on the conjugacy class generated by $x$, $C_{x}=\{g^{-1}xg:g\in G\}
$. In \cite[Thm.~9.1]{GHAdv}, the minimum integer $k(x)$ for which $\mu
_{x}^{k(x)}$ is absolutely continuous was determined for all the classical
Lie groups. The number $k(x)$ depended on the type of the set of
annihilating roots of $x$, where in this setting by the set of annihilating
roots we mean 
\begin{equation*}
\Phi _{x}:=\{\alpha \in \Phi :\alpha (x)\equiv 0\text{ mod }2\pi \}.
\end{equation*}
Again, by the type of $x$, we will mean the type of $\Phi _{x}$.

This was extended by Wright \cite{Wr} to convolution products of (possibly)
different $\mu _{x}$ in the case of $SU(n)$, obtaining the same
characterization as for the Lie algebra problem. In this subsection, we will
obtain a similar result for all the classical Lie groups whenever the group
elements $x_{i}=\exp X_{i}$ where $X_{i}\in \mathfrak{g}$ and $x_{i}\in G$
have the same type.

We need the following preliminary result, analogous to Prop.~\ref{key}.
Given $x\in G$, we let 
\begin{equation*}
\mathcal{N}_{x}:=\{RE_{\alpha },IE_{\alpha }:\alpha (x)\neq 0\text{ mod }%
2\pi \}.
\end{equation*}

\begin{lemma}
\textrm{(}c.f.~\cite{Ra2}, \cite{Wr}\/\textrm{)} The measure $\mu
_{x_{1}}\ast \cdots \ast \mu _{x_{L}}$ on $G_{n}$ is absolutely continuous
with respect to Haar measure on $G_{n}$ if and only if any of the following
hold:

\textrm{(i)} The set $\prod\limits_{i=1}^{L}C_{x_{i}}\subseteq G_{n}$ has
non-empty interior;

\textrm{(ii)} The set $\prod\limits_{i=1}^{L}C_{x_{i}}\subseteq G_{n}$ has
positive measure;

\textrm{(iii)} There exists $g_{i}\in G_{n}$ with $g_{1}=Id$, such that 
\begin{equation*}
sp\{Ad(g_{i})\mathcal{N}_{x_{i}}:i=1,\dots,L\}=\mathfrak{g}_{n}.
\end{equation*}
\end{lemma}

\begin{proposition}
Let $x_{1},\dots,x_{L}\in G_{n}$ and assume $x_{i}=\exp X_{i}$ for some $%
X_{i}\in \mathfrak{g}_{n}$ where $x_{i}$ and $X_{i}$ have the same type.
Then $\mu _{x_{1}}\ast \cdots \ast \mu _{x_{L}}$ is absolutely continuous
with respect to Haar measure on $G_{n}$ if and only if $\mu _{X_{1}}\ast
\cdots \ast \mu _{X_{L}}$ is absolutely continuous with respect to Lebesgue
measure on $\mathfrak{g}_{n}$. Moreover, $\prod\limits_{i=1}^{L}C_{x_{i}}$
has non-empty interior in $G_{n}$ if and only if $\sum_{i=1}^{L}O_{X_{i}}$
has non-empty interior in $\mathfrak{g} _{n} $.
\end{proposition}

\begin{proof}
If $x_{i}$ and $X_{i}$ are of the same type, then $\mathcal{N}_{x_{i}}= 
\mathcal{N}_{X_{i}}$. Consequently, 
\begin{equation*}
sp\{Ad(g_{i})\mathcal{N}_{x_{i}}:i=1,\dots,L\}=sp\{Ad(g_{i})\mathcal{N}
_{X_{i}}:i=1,\dots,L\}
\end{equation*}
and thus $\mu _{x_{1}}\ast \cdots \ast \mu _{x_{L}}$ is absolutely
continuous if and only if $\mu _{X_{1}}\ast \cdots \ast \mu _{X_{L}}$ is
absolutely continuous. The latter statement holds as absolute continuity is
equivalent to non-empty interior of either the product of conjugacy classes
or the sum of orbits, depending on the setting.
\end{proof}

\begin{remark}
If $x_{i}=\exp X_{i}$ and $\mu _{X_{1}}\ast \cdots \ast \mu _{X_{L}}$ is not
absolutely continuous, then it still follows $\mu _{x_{1}}\ast \cdots \ast
\mu _{x_{L}}$ is not absolutely continuous. We simply note that always $\Phi
_{X_{i}}\subseteq \Phi _{x_{i}}$.
\end{remark}

Consider the Lie group $SU(n)$. Every conjugacy class contains a diagonal
matrix so in studying the measure $\mu _{x}$ there is no loss in assuming 
\begin{equation*}
x=diag(\exp ia_{1},\dots,\exp ia_{n})
\end{equation*}
where $a_{j}\in [0,2\pi )$ and $\sum a_{j}\equiv 0\text{ mod }2\pi $. Notice
that $x=\exp X$ where $X=diag(ia_{1},\dots.,ia_{n})$ belongs to $su(n)$. The
root $\alpha =e_{j}-e_{k}$ acts on $x$ (and $X$) by $\alpha (x)=a_{j}-a_{k}$%
. Thus $\Phi _{x}=\Phi _{X}$ and so the Proposition applies to all $L$%
-tuples in $SU(n)$.

This is not true for the other classical Lie groups. For example, in $%
SO(2n+1)$ (type $B_{n})$ there is an element $x$ with $\Phi _{x}=\{e_{i}\pm
e_{j}:1\leq i\neq j\leq n\}$, i.e., of type $D_{n}$. This type does not
arise in the Lie algebra. Indeed, the only element $X\in so(2n+1)$ with $%
\Phi _{X}\supseteq \Phi _{x}$ is $X=0$. The element $x$ has the property
that $\mu _{x}^{2n}\in L^{1}(G)$, but $\mu _{x}^{2n-1}$ is singular with
respect to Haar measure on $G$. In contrast, any $X$ with $\exp X=x$ has $%
\mu _{X}^{2}\in L^{2}(\mathfrak{g})$.

These additional (and often more complicated) types of elements that can
arise in the Lie groups make the problem of characterizing absolute
continuity of orbital measures on Lie groups more challenging than for Lie
algebras.


\begin{thebibliography}{99}
\bibitem{DRW} A.~Dooley, J.~Repka and N.~Wildberger, \textit{Sums of adjoint
orbits}, Linear and multilinear algebra \textbf{36}(1993), 79--101.

\bibitem{FG} A.~Frumkin and A.~Goldberger, \textit{On the distribution of
the spectrum of the sum of two Hermitian or real symmetric spaces}, Adv.\
Appl.\ Math.\ \textbf{37}(2006), 268--286.

\bibitem{GSJGA} P.~Gratczyk and P.~Sawyer, \textit{On the kernel of the
product formula on symmetric spaces}, J.\ Geom.\ Anal.\ \textbf{14}(2004),
653--672.

\bibitem{GSJFA} P.~Gratczyk and P.~Sawyer, \textit{Absolute continuity of
convolutions of orbital measures on Riemannian symmetric spaces}, J.\ Func.\
Anal.\ \textbf{259}(2010), 1759--1770.

\bibitem{GSLie} P.~Gratczyk and P.~Sawyer, \textit{A sharp criterion for the
existence of the density in the product formula on symmetric spaces of Type} 
$A_{n}$, J.\ Lie Theory \textbf{20}(2010), 751--766.

\bibitem{GAFA} S.K.~Gupta and K.E.~Hare, \textit{Singularity of orbits in
classical Lie algebras}, Geom.\ Func.\ Anal.\ \textbf{13}(2003), 815--844.

\bibitem{GHGeneric} S.K.~Gupta and K.E.~Hare, \textit{Convolutions of
generic orbital measures in compact symmetric spaces}, Bull.\ Aust.\ Math.\
Soc.\ \textbf{79}(2009), 513--522.

\bibitem{GHAdv} S.K.~Gupta and K.E.~Hare, $L^{2}$\textit{-singular dichotomy
for orbital measures of classical compact Lie groups}, Adv.\ Math.\ \textbf{%
222}(2009), 1521--1573.

\bibitem{GHMathZ} S.K.~Gupta, K.E.~Hare and S.~Seyfaddini, $L^{2}$\textit{\
-singular dichotomy for orbital measures of classical simple Lie algebras},
Math.\ Zeit.\ \textbf{262}(2009), 91--124.

\bibitem{HJY} K.E.~Hare, D.~Johnstone, F.~Shi and M.~Yeung, $L^{2}$\textit{%
-singular dichotomy for exceptional Lie groups and algebras}, J.\ Aust.\
Math.\ Soc.\ \textbf{95}(2013), 362--382.

\bibitem{He} S.~Helgason, \textit{Differential geometry, Lie groups and
symmetric spaces}, Academic Press, New York, 1978.

\bibitem{Hu} J.~Humphreys, \textit{Introduction to Lie algebras and
representation theory}, Springer-Verlag, New York, 1972.

\bibitem{La} S.~Lang, \textit{Differential and Riemannian manifolds},
Springer-Verlag, New York, 1995.

\bibitem{Kn} A.~Knapp, \textit{Lie groups beyond an introduction},
Birkhauser, Verlag AG, 2002.

\bibitem{KT} A.~Knutson and T.~Tao, \textit{Honeycombs and sums of Hermitian
matrices}, Notices Amer.\ Math.\ Soc.\ \textbf{48}(2001), 175--186.

\bibitem{MT} M.~Mimura and H.~Toda, \textit{Topology of Lie groups}, Trans.\
Math.\ Monographs 91, Trans.\ Amer.\ Math.\ Soc., Providence, R.I., 1991.

\bibitem{Ra} D.~Ragozin, \textit{Zonal measure algebras on isotropy
irreducible homogeneous spaces}, J.\ Func.\ Anal.\ \textbf{17}(1974),
355--376.

\bibitem{Ra2} D.~Ragozin, \textit{Central measures on compact simple Lie
groups}, J.\ Func.\ Anal.\ \textbf{10}(1972), 212--229.

\bibitem{RS} F.~Ricci and E.~Stein, \textit{Harmonic analysis on nilpotent
groups and singular integrals. II. Singular kernels supported on
submanifolds }, J.\ Func.\ Anal.\ \textbf{78}(1988), 56--84.

\bibitem{RS1} F.~Ricci and E.~Stein, \textit{Harmonic analysis on nilpotent
groups and singular integrals. III. Fractional integration along manifolds,}
J.\ Func.\ Anal.\ \textbf{86}(1989), 360--389.

\bibitem{Va} V.S.~Varadarajan, \textit{Lie groups and Lie algebras and their
representations}, Springer-Verlag, New York, 1984.

\bibitem{W} N.~Wildberger, \textit{On a relationship between adjoint orbits
and conjugacy classes of a Lie group}, Can.\ Math.\ Bull.\ \textbf{33}%
(1990), 297--304.

\bibitem{Wr} A.~Wright, \textit{Sums of adjoint orbits and $L^{2}$-singular
dichotomy for }$SU(m)$, Adv.\ Math.\ \textbf{227}(2011) 253--266.
\end{thebibliography}
\end{document}